\documentclass[10pt,a4paper]{amsart}
\usepackage{verbatim}

\usepackage[T1]{fontenc}
\usepackage{graphicx}
\usepackage{enumerate}
\usepackage{amsmath,amsfonts,amssymb}
\usepackage{ulem}
\usepackage{color}
\usepackage{cite}
\definecolor{citation}{rgb}{0.2,0.58,0.2} 
\definecolor{formula}{rgb}{0.1,0.2,0.6}
\definecolor{url}{rgb}{0.3,0,0.5}
\usepackage{pgf,tikz}
\usepackage{mathrsfs}
\usepackage{fancyhdr}
\usepackage{dsfont}

\title[Wolff potentials in elliptic problems with  Orlicz growth]{Wolff potentials and local behaviour\\ of solutions to measure data elliptic problems\\ with  Orlicz growth}

\author{Iwona Chlebicka}\address{Iwona Chlebicka \\
Institute of Applied Mathematics and Mechanics, University of Warsaw \\ ul. Banacha 2, 02-097 Warsaw, Poland\\  \texttt{e-mail: i.chlebicka@mimuw.edu.pl}}
\author{Flavia Giannetti} \address{Flavia Giannetti \\
Dipartimento di Matematica e Applicazioni \lq\lq R.
Caccioppoli\rq\rq \\ Universit\`a  degli Studi di Napoli\\
Via Cintia, 80126 Napoli,
Italia\\ 
\texttt{e-mail: giannett@unina.it}}
\author{Anna Zatorska-Goldstein}\address{Anna Zatorska-Goldstein \\ Institute of Applied Mathematics and Mechanics, University of Warsaw \\ ul. Banacha 2, 02-097 Warsaw, Poland \\ \texttt{e-mail: azator@mimuw.edu.pl}}

\date{}
\begin{document}
\maketitle \sloppy

\thispagestyle{empty}

\belowdisplayskip=18pt plus 6pt minus 12pt \abovedisplayskip=18pt
plus 6pt minus 12pt
\parskip 4pt plus 1pt
\parindent 0pt

\newcommand{\barint}{
         \rule[.036in]{.12in}{.009in}\kern-.16in
          \displaystyle\int  } 
\def\R{{\mathbb{R}}}
\def\rp{{[0,\infty)}}
\def\r{{\mathbb{R}}}
\def\n{{\mathbb{N}}}
\def\l{{\mathbf{l}}}
\def\bu{{\bar{u}}}
\def\bg{{\bar{g}}}
\def\bG{{\bar{G}}}
\def\ba{{\bar{a}}}
\def\bv{{\bar{v}}}
\def\bmu{{\bar{\mu}}}
\def\rn{{\mathbb{R}^{n}}}
\def\rN{{\mathbb{R}^{N}}} 

\newtheorem{theo}{\bf Theorem} 
\newtheorem{coro}{\bf Corollary}[section]
\newtheorem{lem}[coro]{\bf Lemma}
\newtheorem{rem}[coro]{\bf Remark} 
\newtheorem{defi}[coro]{\bf Definition} 
\newtheorem{ex}[coro]{\bf Example} 
\newtheorem{fact}[coro]{\bf Fact} 
\newtheorem{prop}[coro]{\bf Proposition}

\newcommand{\dv}{{\rm div}}
\def\aI{\texttt{(a1)}}
\def\aII{\texttt{(a2)}}
\newcommand{\MBm}{{(M_B^-)}}
\newcommand{\MBmj}{{(M_{B_{r_j}}^-)}}
\newcommand{\aBi}{{a^B_{\rm i}}}
\newcommand{\opA}{{\mathcal{ A}}}
\newcommand{\wt}{\widetilde}
\newcommand{\ve}{\varepsilon}
\newcommand{\vp}{\varphi}
\newcommand{\vt}{\vartheta}
\newcommand{\gb}{{g_\bullet}}
\newcommand{\gbn}{{(\gb)_n}}
\newcommand{\vr}{\varrho}
\newcommand{\pa}{\partial}
\newcommand{\cW}{{\mathcal{W}}}
\newcommand{\supp}{{\rm supp}}
\newcommand{\miu}{{\min_{\partial B_k}u}}

\newcommand{\data}{\textit{\texttt{data}}}


\parindent 1em

\begin{abstract}
We establish pointwise estimates expressed in terms of a nonlinear potential of a generalized Wolff type for $\opA$-superharmonic functions with nonlinear operator $\opA:\Omega\times\rn\to\rn$ having measurable dependence on the spacial variable and  Orlicz growth with respect to the last variable.  The result is sharp as the same potential controls bounds from above and from below. Applying it we provide a bunch of precise regularity results including continuity and H\"older continuity for solutions to problems involving measures that satisfies conditions expressed in the natural scales. Finally, we give a variant of Hedberg--Wolff theorem on characterization of the dual of the Orlicz space.
\end{abstract}


\section{Introduction}

Potential estimates are well-established and precise tools in analysis of measure data nonlinear elliptic partial differential equations~\cite{KuMi2014}. Being influenced by~\cite{KiMa92,KiMa94} settling  
the nonlinear potential theory for $p$-superharmonic functions related to the Dirichlet problem $-\Delta_p u=-\dv (|Du|^{p-2}Du)=0$, $1<p<\infty$,
and noticing the special place of Orlicz growth measure data problems in recent nonlinear analysis~\cite{Baroni-Riesz,bemi,Byun3,CO,
IC-pocket,CiMa,DiMa,KuMi2014,Marc2020}, we aim  to study   nonstandard growth version of pointwise estimates involving suitably generalized potential of the Wolff type. The estimates are powerful and have multiple new regularity consequences for local behaviour of very weak solutions to problems involving measure data that satisfy conditions expressed in the relevant scales of generalized Lorentz, Marcinkiewicz, or Morrey type. On the other hand, as an another application of the potential estimates, we provide the proof of Orlicz version of Hedberg--Wolff Theorem yielding full characterization of the  natural dual space to the space of solutions by the means of the  Wolff potential.

Let us stress a few highlights. We treat whole the range of doubling growths within one attempt. In particular, subquadratic and superquadratic cases are included. In turn, we cover the case of $p$-Laplacian, $1<p<\infty$, possibly with measurable coefficients, but unlike the classical studies~\cite{hekima,KiMa92} the operator we consider is {\em not} assumed to enjoy homogeneity of a~form $\opA(x,k\xi)=|k|^{p-2}k\opA(x,\xi)$. Consequently, our class of solutions is {\em not} invariant with respect to scalar multiplication. We do {\em not} follow the scheme of the proof of~\cite{KiMa92}, \cite{Maly-Orlicz} nor~\cite{LuMaMa}. Our main inspirations are~\cite{KoKu,tru-wa} and~\cite{KuMi2014}. The tools of potential analysis used extensively in the proof have been elaborated lately in~\cite{CZG-gOp} for $\opA$-superharmonic problems with more general growth of Musielak-Orlicz type.
 
We investigate $\opA$-superharmonic functions and $\opA$-supersolutions to measure data problem for nonlinear operators $\opA:\Omega\times\rn\to\rn$ having Orlicz growth with respect to the second variable and built upon the model case
 \begin{equation}
\label{intro:eq:main} -{ \dv}\left(a(x)\frac{G(|D u|)}{|D u|^2}  D u \right)=\mu\quad\text{in}\quad \Omega,
\end{equation}
where $ a\in L^\infty(\Omega)$ is separated from zero, $\mu$ is a nonnegative measure, and $G\in\Delta_2\cap\nabla_2$, 
 see {\rm Assumption {\bf (A)}} in Section~\ref{sec:formulation-n-reg-cons}. The subtlety in distiguishing these types of functions is exposed in detail in Sections~\ref{ssec:sols} and~\ref{ssec:A-sh-gen-meas}. Briefly,  $\opA$-superharmonic functions are defined by the Comparison Principle with respect to a family of continuous solutions to $-\dv\opA(x,Du)=0$  for $\opA:\Omega\times\rn\to\rn$, such that $x\mapsto\opA(x,\cdot)$ is merely measurable and bounded, whereas $\xi\mapsto\opA(\cdot,\xi)$ is continuous and governed by a doubling Orlicz function. 
{Observe that we} cover the case of $p$-Laplacian when $G_p(s)=s^{p},$ for every $p>1$, together with operators governed by the Zygmund-type functions $G_{p,\alpha}(s)=s^p\log^\alpha(1+s)$, $p>1,\,\alpha\geq 0$, as well as their multiplications and compositions with various parameters.

  Let us remark that since weak solutions does not have to exist for arbitrary measure datum, we employ a notion of very weak solutions obtained by an approximation. The definition coincides with SOLA (Solutions Obtained as a Limit of Approximation) introduced in~\cite{BG} except for the use of truncations~\eqref{Tk} as in~\cite{bbggpv}, which broadens the class of admissible solutions. See Section~\ref{ssec:sols} for the precise definition of this notion of very weak solutions, \cite{CiMa} for the existence and basic regularity, and~\cite{ACCZG,CGZG,IC-measure-data} for related existence results. Such solutions can be unbounded if the measure concentrates (see Corollary~\ref{coro:potential} and Remark~\ref{rem:p-potential} for examples), though still one can provide pointwise estimates by a~relevant potential from above and below. In our study we include the celebrated case of $p$-superharmonic functions, for which Wolff-potential estimates were provided by Kilpel\"ainen and Mal\'y in~\cite{KiMa92,KiMa94} and reproved in~\cite{KoKu}. In fact, it is established that a $p$-superharmonic function $u$  satisfies  {estimates}
\begin{equation}
    \label{wolff-p-est}
\tfrac 1c\cW_p^{\mu_u}(x_0,R)\leq u(x_0)\leq c\Big(\inf_{B(x_0,R)}u+\cW_p^{\mu_u}(x_0,R)\Big)\quad\text{for some }\ c=c(n,p)
\end{equation}with {the} so-called Wolff potential
\[ \cW_p^{\mu_u}(x_0,R)=\int_0^R \left(r^{p-n}{\mu_u(B(x_0,r))}\right)^{\frac{1}{p-1}}\,\frac{dr}{r}= \int_0^R \left(\frac{\mu_u(B(x_0,r))}{r^{n-1}}\right)^{\frac{1}{p-1}}\,dr.\]
In the linear case ($p=2$) estimates~\eqref{wolff-p-est} retrieve the classical Riesz potential bounds. Potentials of this type are investigated since~\cite{HaMa2,Ma} stemming from studies on Wiener criterion for $p$-Laplacian. See~\cite{KuMi2014} for an overview, \cite{HedWol,KiMa92,KiMa94,LiMa} for results of fundamental meaning for the theory,~\cite{adams-hedberg,hekima,BB} for well-present background, and~\cite{PV1,PV2} for further consequences of Wolff potential estimates in the theory of existence. The preeminent role of estimates of a form~\eqref{wolff-p-est} is played in the regularity theory. See \cite{KiMa92,KiMa94} for the classical results embraced and extended in Section~\ref{ssec:loc-beh}.

 The theorem of Hedberg and Wolff proven first in~\cite{HedWol} characterizes the dual to the Orlicz-Sobolev space. In fact, it yields that for a nonnegative measure compactly supported in $\Omega$ it holds that \begin{equation}
     \label{wolff-hed-thm} \text{$\mu\in W^{-1,p'}(\Omega)\ \ $ if and only if }\ \int_\Omega \mathcal{W}^\mu_p(x,R)\,d\mu(x)<\infty\ \text{ for some }\ R>0.
 \end{equation}

Since our goal is to provide estimates of a type~\eqref{wolff-p-est} and a version of~\eqref{wolff-hed-thm} for a~possibly broad class of degenerate operators, let us give an overview on related results available with relaxed growth. The first generalization of~\cite{KiMa92} to the weighted case has been elaborated in~\cite{Mik}.  Wolff-potential estimates for the problems stated in the variable exponent setting are provided and applied further in regularity theory~\cite{LuMaMa,BaHa,RnW}. Similar estimates in the Orlicz setting  were available before~\cite{Maly-Orlicz,FioP} and used in the theory of nonexistence~\cite{FioGia2,FioP}. Potential estimates imply local regularity of solutions, for a~different approach than the mentioned one see~\cite{Ci-pot}.

The statements of our results can be found in Section~\ref{sec:formulation-n-reg-cons}. Theorem~\ref{theo:est-sol} yields a~result generalizing~\eqref{wolff-p-est}. Namely, for $\opA$-superharmonic functions related to distributional solutions to~\eqref{intro:eq:main}, we get estimates of a form~\eqref{wolff-p-est} with the use of the following version of Wolff potential
\begin{equation}
\label{Wolff-potential}
\mathcal{W}^{\mu_u}_G(x_0,R)=\int_0^R g^{-1}\left(\frac{\mu_u(B(x_0,r))}{r^{n-1}}\right)\,dr,
\end{equation}
where $g(s)=G'(s)$ (where for $p$-Laplace case we have $g(s)= \frac{1}{p}s^{p-1}=(s^p)'$. Further, in Section~\ref{ssec:loc-beh} we present a bunch of regularity consequences of these potential estimates. In particular, we find conditions on the measure datum formulated in the natural scales generalizing Morrey, Lorentz, and Marcinkiewicz scales, which ensure continuity or H\"older continuity of $\opA$-superharmonic functions essentially extending the known ones~\cite{CL,ZFZ}. The results are precise since the conditions are kept fully intrinsic, {namely} at no stage we compare the growth to the polynomial one. Our second accomplishment is Theorem~\ref{theo:W-H} given in Section~\ref{ssec:WH} and yielding the nonstandard growth version of theorem of Hedberg and Wolff. It states that for a nonnegative measure $\mu$ compactly supported in $\Omega$ it holds that \[\text{$\mu\in (W^{1,G}_0(\Omega))'\ \ $ if and only if }\ \int_\Omega \mathcal{W}^\mu_G(x,R)\,d\mu(x)<\infty\quad\text{for some }\ R>0.\] 

Similar result to our upper pointwise bound from Theorem~\ref{theo:est-sol} (generalizing~\eqref{wolff-p-est}) for related variational problem exposing Orlicz of the type
\[w\mapsto{\mathcal{G}}(w)=\int_\Omega a(x)G(|D w|)\,dx\]  growth was provided by Mal\'y~\cite{Maly-Orlicz}. 
Let us stress that not only we prove both lower and upper bounds, but also we infer multiple consequences of our main result. Moreover, we use  completely other arguments than engaged in~\cite{Maly-Orlicz}. We also do not follow the classical way of~\cite{KiMa92}, nor ideas of~\cite{LuMaMa} provided for the variable exponent counterpart of $p$-Laplacian. In fact, for the method of our reasoning the most influential was later proof provided by Korte and Kuusi~\cite{KoKu} based on the ideas by Trudinger and Wang~\cite{tru-wa}. Similar scheme was  used also in~\cite{hara,lu}. This approach requires a basic toolkit of potential analysis provided in~\cite{CZG-gOp}.

Apart from multiple sharp local regularity consequences of Theorem~\ref{theo:est-sol} given in Section~\ref{sec:formulation-n-reg-cons}, this result can be used further to prove so-called Wiener criterion giving necessary and sufficient geometric conditions for boundary points to guarantee continuity up to the boundary whenever boundary values are continuous at that point, see~\cite{KiMa94}. On the other hand, since Wolff potential estimates can be used also to infer gradient regularity for measure data problems~\cite{DuMi',DuMi,KuMi2014,KuMi2018,NP}, it would be worth to study them in the context of Orlicz-growth problems and to compare to the known ones, cf.~\cite{bemi,Baroni-Riesz,Byun1,Byun2,IC-gradest,IC-lower,Yao-Zheng}. Favorable versions of Theorems~\ref{theo:est-sol} and~\ref{theo:W-H} in the generality of~\cite{CZG-gOp} would be also very interesting.

The paper is organized as follows. Our assumptions, main results and their consequences in local behaviour of~$\opA$-superharmonic functions are presented in Section~\ref{sec:formulation-n-reg-cons}. Section~\ref{sec:prelim} is devoted to notation and information on the settings, as well as recalling the basal information on the nonstandard growth potential theory. Section~\ref{sec:mainproof} contains the proof of the potential estimates of Theorem~\ref{theo:est-sol}, whereas Section~\ref{sec:W-H} gives the proof of the Hedberg--Wolff theorem.

\section{The main result and its consequences }\label{sec:formulation-n-reg-cons}

\subsection{The statement of the problem and the main result}
Let us give the details about the problem we study. $\opA$-superharmonic functions are defined in Section~\ref{ssec:sols} by the Comparison Principle with respect to $\opA$-harmonic functions, i.e. continuous solutions to \begin{equation}
    \label{eq:0}
-\dv\opA(x,Du)= 0
\end{equation}
with a vector field satisfying the following conditions.

\subsubsection*{\underline{Assumption {\bf (A)}}}  Given a bounded open set $\Omega\subset\rn$, $n\geq 2$, let us consider a~Carath\'eodory's function $\opA:\Omega\times \rn\to\rn$, which is monotone, that is
\[(\opA(x,\xi)-\opA(x,\eta))\cdot(\xi-\eta)>0\qquad\text{for every } \ \xi\neq\eta.\]
We suppose that $\opA$ satisfies growth and coercivity conditions expressed by the means of a doubling $N$-function $G\in C^1 {((0,\infty))}$ i.e. a nonnegative, increasing, and convex function such that $G\in\Delta_2\cap\nabla_2$. Namely, we assume that
\begin{equation}
\label{ass-op}\begin{cases}
c_1^\opA G(|\xi|)\leq \opA(x,\xi)\cdot\xi,\\
|\opA(x,\xi)|\leq c_2^\opA g(|\xi|),
\end{cases}
\end{equation}
where $g$ is the derivative of $G$ and $c_1^\opA,c_2^\opA>0$ are absolute constants. We collect all parameters of the problem as $ \data=\data(i_G,s_G,c_1^\opA,c_2^\opA). $

Definitions of very weak solutions called `approximable solutions', $\opA$-supersolutions, $\opA$-harmonic functions etc. are given below the introduction to the functional setting, that is in Section~\ref{ssec:sols}. {In Section~\ref{ssec:A-sh-gen-meas} we explain in what sense an $\opA$-superharmonic function generates a measure.}

\subsubsection*{\underline{Wolff potential estimates}} Our main accomplishment is the following pointwise estimate.
\begin{theo}\label{theo:est-sol} 
Suppose a vector field $\opA:\Omega\times\rn\to\rn$ satisfies {\rm Assumption {\bf (A)}} with an $N$-function $G\in\Delta_2\cap\nabla_2$.  There exists $R_W >0$ such that if  $u$ is a nonnegative,  $\opA$-superharmonic function in $B(x_0,{R_\mathcal{W}}) \Subset\Omega$ which is finite a.e., and if $\mu_u$ is generated by $u$,  then for $R\in(0,R_\mathcal{W}/2)$ we have
\begin{flalign}
\label{eq:u-est}  C_L\left(\mathcal{W}^{\mu_u}_G (x_0 , R) - R\right)\leq u(x_0 )&\leq C_U\left( \inf_{B(x_0,R)} u(x)+\mathcal{W}^{\mu_u}_G (x_0 , R)+R\right)
\end{flalign}
with $C_L,C_U>0$ depending only on $\data$ and $n$.
\end{theo}

\begin{rem}The radius $R<R_\cW$ has to satisfy a smallness condition from Harnack's inequality (Proposition~\ref{prop:harnack}), namely $R<R_H(n)$ and it has to be such that $B_{3R}\Subset\Omega$ and $\vr_{g,B_{3R}}(|Du|)\leq 1$, cf. Remark~\ref{rem:g-intergrability-of-Ah}.
\end{rem}

A remarkable point here is to realize that the same potential controls both estimates, so it cannot be replaced by any other smaller potential. In turn, the estimate~\eqref{eq:u-est} is essentially sharp. Note also that it is not essential that on the right-hand side we have infimum of $u$. Without loss of the strength of the result it could be substituted by some other finite quantity.

Potential estimates are known to be an efficient tool to bring precise information on the local behaviour of solutions. We refer to~\cite{KuMi2014} for clearly presented overview of consequences of estimates like~\eqref{eq:u-est} in studies on $p$-superharmonic functions. Below we present applications of Theorem~\ref{theo:est-sol} in the regularity theory and in the proof of the Hedberg--Wolff Theorem for operators of general growth.

\subsection{Local behaviour of solutions to measure data problems}\label{ssec:loc-beh} In this section we investigate the  great deal of powerful consequences of Theorem~\ref{theo:est-sol} in the regularity of $\opA$-superharmonic functions related to `approximable solutions' $u$ to a measure data problems \[-\dv \opA(x,Du)=\mu\geq 0,\]
see Sections~\ref{ssec:sols} and~\ref{ssec:A-sh-gen-meas}. 
Note that the conditions are expressed in the very natural scales and the results were not known in this generality before.

Let us start with  the following trivial remark.
\begin{coro} Under {\rm Assumption {\bf (A)}} suppose $u$ is a nonnegative $\opA$-superharmonic function in $\Omega$ and $\mu_u:=-\dv \opA(x,Du)$ in the sense of distributions. Then $u$ is locally bounded if and only if $\cW^\mu_G(\cdot,R)$ is locally bounded for all sufficiently small $R$.
\end{coro}
Note that smallness of the potential is equivalent to continuity of the solution.
\begin{coro} \label{coro:cont}
 Under {\rm Assumption {\bf (A)}} suppose $u$ is a nonnegative $\opA$-superharmonic and finite a.e. in $\Omega$ and $\mu_u:=-\dv \opA(x,Du)$ in the sense of distributions. Then $u$ is continuous in $x_0$ if and only if for every $\ve>0$ there exists $r>0$, such that $\cW^{\mu_u}_G(x,r)<\ve$ whenever $x\in B(x_0,r).$
\end{coro}
\begin{proof} 
We suppose 
that $u$ is continuous in $x_0$. Let us fix $\ve$ and choose $r\in (0,\ve)$, such that $|u-u(x_0)|<\ve C_L$ in ${B(x_0,4r)}\Subset\Omega.$ Set $B=B(x,r)$ for some $x\in B(x_0,r)$. We observe that $(u-\inf_{B} u)$ is a nonnegative, $\opA$-superharmonic function and $\mu_u=\mu_{(u-\inf_B u)}$ (see Section~\ref{ssec:A-sh-gen-meas}), so we can apply Theorem~\ref{theo:est-sol} to get \[\cW^{\mu_u}_G(x,r)\leq \tfrac{1}{C_L}(u(x)-\inf_Bu)+r=  \tfrac{u(x)-u(x_0)}{C_L}+\tfrac{u(x_0)-\inf_Bu}{C_L}+r\leq 3{\ve}.\]

On the other hand, for the reverse suppose $u(x_0)<\infty$. Since $u$ is lower semicontinuous, for any $\ve>0$ we may choose $r\in(0,\ve)$ such that $u>u(x_0)-\ve$ in $B(x_0,4r)$. Moreover, $(u-u(x_0)+\ve)$ is $\opA$-superharmonic in $B(x_0,4r)$ and generates the same measure as $u$. Observe that  for every $x\in B(x_0,r)$ we have $B(x_0,r)\subset B(x,2r)\subset B(x_0,4r).$ Then {function $(u-u(x_0)+\ve)$ is $\opA$-superharmonic in $B(x,2r)$, thus}  for every $x\in B(x_0,r)$ we infer {by Theorem \ref{theo:est-sol}}
\begin{flalign*}
0&<u(x)-u(x_0)+\ve\leq C_U\Big(\inf_{{B(x,2r)}}(u-u(x_0)+\ve)+ {2r}+\cW^{\mu_u}_G(x, {2r})\Big)\\
&\leq C_U\Big(\inf_{B(x_0,r)}(u-u(x_0)+\ve)+ {2r}+\cW^{\mu_u}_G(x, {2r})\Big)\leq C_U(0+ 3)\ve,
\end{flalign*}
which ends the proof.
\end{proof}

Corollary~\ref{coro:cont} has the following direct consequence.
\begin{rem}\rm If $u$ and $v$ are nonnegative and $\opA$-superharmonic in $\Omega$, $u$ is continuous in $x_0\in\Omega$, and for some absolute constant $c>0$ it holds $\mu_v\leq c\mu_u$, then $v$ is also continuous in $x_0$.\end{rem}

Let us concentrate on the estimate on the counterpart of the fundamental solution, retrieving the optimal conditions in the power-growth case, see Remark~\ref{rem:p-potential}.
\begin{coro}\label{coro:potential}
 Under {\rm Assumption {\bf (A)}} suppose $u$ is a nonnegative $\opA$-superharmonic function in $B(x_0, R_\cW)\Subset\Omega$, such that $-\dv\opA(x,Du)=\mu_u=\delta_{x_0}$  in the sense of distributions. Assume further that $x$ is  close to $x_0$, namely such that for $r=|x-x_0|$ there holds $B_{r}=B(x,r)\subset B(x,2r) \subset B(x_0, R_\cW/4)$. Then there exists $c=c(\data,n)>0$, such that \[c^{-1}\left(\int_r^{2r} g^{-1}\left({s^{1-n}}\right)\,ds-r\right)\leq u(x)\leq c\left( \int_r^{2r} g^{-1}\left({s^{1-n}}\right)\,ds+\inf_{B_{2r}} u+r\right).\]
 If additionally $G$ is so fast {at} infinity that \begin{equation}
    \label{int-div}\int_0 g^{-1}\left({s^{1-n}}\right)\,ds<\infty,\end{equation}   then $u\in L^\infty(B_r).$ This bound is optimal.%
\end{coro}
  
\begin{proof} Since $u$ is $\opA$-superharmonic  in $B(x, 2r)$, we  have by  Theorem~\ref{theo:est-sol}
\begin{flalign*}  C\left(\mathcal{W}^{\mu_u}_G (x , 2r) - 2r\right)\leq u(x )&\leq C\left( \inf_{B_{2r}} u(x)+\mathcal{W}^{\mu_u}_G (x , 2r)+2r\right).
\end{flalign*}
On the other hand, recalling that $\mu_u=\delta_{x_0}$, we get
$$\mathcal{W}^{\mu_u}_G (x , 2r)=\int_r^{2r} g^{-1}\left({s^{1-n}}\right)\,ds$$ and hence the estimate from the claim follows.

Assumption \eqref{int-div} is precisely the one ensuring optimal embedding of Orlicz-Sobolev spaces into $L^\infty$ or into the space of continuous bounded functions, see \cite{Ci96-emb,Ci-SNS,Ci-some}. Indeed, it suffices to recall that $ g^{-1}(t)\simeq (\wt G)'(t)$ and $\wt G(t)\simeq t (\wt G)'(t)$ to get for $0<a<b$ that
\begin{equation*}
\int_a^{b} g^{-1}(s^{1-n})\,ds \simeq \int_a^{b} (\wt G)' (s^{1-n})\,ds \simeq \int_a^{b} \wt G(s^{1-n})s^{n-1}\,ds \simeq  \int_{b^{1-n}}^{a^{1-n}} \frac{\wt G(t)}{t^{1+n'}}\,dt.
\end{equation*}
 Therefore,~\eqref{int-div} is equivalent to \begin{equation*}
     \int^\infty \frac{\wt G(t)}{t^{1+n'}}\,dt<\infty.\end{equation*}
\end{proof}
\begin{rem}\rm\ {
If $u$ is an $\opA$-superharmonic function which solves the equation $-\dv\opA(x,Du)=\mu$ with a bounded Radon measure $\mu$ in the distributional sense, and \eqref{int-div} holds, then $u$ is bounded. It follows that $u \in W^{1,G}_0(\Omega)$ (see e.g. \cite{CZG-gOp}), and therefore $\mu \in (W^{1,G}_0(\Omega))'$. As a matter of fact, then the distributional solutions are weak solutions and the classical regularity theory provide more information by completely other strong tools.} 
\end{rem}

\begin{rem}\label{rem:p-potential}\rm When $G(t)\simeq t^p,$ $1<p<\infty,$ by Corollary~\ref{coro:potential} we retrieve the classical estimates for the fundamental solution to $p$-Laplace equation $-\Delta_pu=\delta_0.$ In fact, if $u$ is a nonnegative $\opA$-superharmonic function satisfying $-\dv \opA(x,Du)=-{ \dv}\left(a(x)|D u|^{p-2} D u \right)=\delta_0$ in the sense of distributions for $0<a\in L^\infty(\Omega)$  separated from zero, then \begin{flalign*}
   c^{-1} |x|^{-\frac{n-p}{p-1}}\leq u(x)&\leq c\left(|x|^{-\frac{n-p}{p-1}}+\inf_{{B(x,2|x|)}} u\right)\quad \text{ when }\  1<p<n,&\\
   c^{-1}\log |x|\leq u(x)&\leq c\left(\log |x|+\inf_{{B(x,2|x|)}} u\right)\quad \ \  \text{ when }\  p=n,&\\
   u(x)&\leq c \qquad\qquad\qquad\qquad\, \text{ when }\  p>n.&\\
 \end{flalign*} 
 Moreover, similarly  Corollary~\ref{coro:potential} has one less precise, but more direct consequence for the general growth operators expressed by the use of indices $i_G,s_G$ defined in~\eqref{iG-sG} for which it holds~\eqref{comp-i_G-s_G}. In fact, when $i_G\neq n\neq s_G$ we have
$$
c^{-1} |x|^{-\frac{n-s_G}{s_G - 1}} \leq u(x) \leq c (|x|^{-\frac{n-i_G}{i_G - 1}} + \inf_{{B(x,2|x|)}} u ).
$$
\end{rem}
{
The above remark can be sharply extended to the Zygmund case.
\begin{rem}\label{rem:O-potential}  \rm Suppose that $1<p<n,$ $\alpha\in\R$, $0<a\in L^\infty(\Omega)$  separated from zero, and $u$ is a nonnegative $\opA$-superharmonic function in $\Omega$, such that
\begin{equation}
    \label{eq-log}-\dv \opA(x,Du)=-{ \dv}\left(a(x)|D u|^{p-2}\log^\alpha({\rm e}+|D u|)  D u \right)=\delta_0
\end{equation}  in the sense of distributions. Then \begin{flalign*}
   c^{-1} |x|^{-\frac{n-p}{p-1}}\log^{-\frac{\alpha}{p-1}}({\rm e}+|x|)\leq u(x)&\leq c\left(|x|^{-\frac{n-p}{p-1}}\log^{-\frac{\alpha}{p-1}}({\rm e}+|x|)+\inf_{{B(x,2|x|)}} u\right).
 \end{flalign*} 
\end{rem}
 }

We can infer local continuity of a solution if the datum belongs to a generalized Lorentz space. This fact has  the best possible consequence in the $p$-Laplace case, see Remark~\ref{rem:p-Lorentz}. Note that this result is meaningful only when $G$ grows so slowly that~\eqref{int-div} is violated, cf.~\cite{Ci-SNS} and Remark~\ref{rem:O-potential} above.

For the reader's convenience we recall some definitions. 
We define the decreasing rearrangement $f^\ast$ of a measurable function $f:\Omega\to\R$ by
$$
f^\ast(t) = \sup \{ s \geq 0 \colon |\{x\in \R^n:f(x)>s\}| > t  \},
$$
the maximal rearrangement by
$$
f^{\ast \ast}(t) =\frac 1t \int_0^t f^\ast(s) \, ds\quad\text{and}\quad 
f^{\ast \ast}(0)=
f^{\ast}(0),
$$
and finally the Lorentz space $L(\alpha, \beta)(\Omega)$ for $\alpha,\beta>0$ as the space of measurable functions such that  
$$ 
 \int_0^\infty \left( t^{1/ \alpha} f^{\ast \ast}(t)\right)^\beta\,\frac{dt}{t} 
 <\infty
$$
(see \cite[Theorem 3.21]{stein-w}).

\begin{coro} \label{coro:cont-2}  Under {\rm Assumption {\bf (A)}} suppose $u$ is a nonnegative $\opA$-superharmonic function in $\Omega$ and $F_u:=-\dv \opA(x,Du)$  in the sense of distributions. If $F_u$ 
satisfies 
\begin{equation}
    \label{cond-Lor}
 \int_0^\infty  t^\frac{1}{n} g^{-1} {\left(t^\frac{1}{n}F_u^{\ast \ast}(t)\right)}\,\frac{dt}{t}  <\infty\end{equation}
for $\Omega_0\Subset\Omega$, then $u\in C(\Omega_0)$.
\end{coro}
\begin{proof}Since $\Omega_0$ can be covered by a finite number of balls, it suffices to provide the estimates in the localized case only. Set $x\in\Omega_0,$ $R_k=2^{1-k}R$ and $B_k=B(x,R_k)\Subset B(x,R_\cW/2)$ for $k=0,1,\dots\,$. As $F_u$ is taken in a place of a measure in Theorem~\ref{theo:est-sol}, with a slight abuse of notation we will write $F_u(B(x,R_k))=\int_{B_k}F_u(y)\,dy.$ We notice that 
we have
\begin{flalign*}
 W_G^{F_u}(x,R)&=\sum_{k=1}^\infty \int_{R_{k+1}}^{R_k} g^{-1}\left(\frac{F_u(B(x,r))}{r}\right)\,dr\lesssim \sum_{k=1}^\infty  {{R_{k+1}}} g^{-1}\left(\frac{F_u(B(x,R_k))}{R_k}\right)\,\\&
 {\lesssim \sum_{k=1}^\infty  {R_{k}} g^{-1}\left(\frac{F_u(B(x,R_k))}{R_k^{n-1}}\right)}
 \end{flalign*}
To estimate the series we employ the decreasing rearrangement $F_u^*$ of $F_u$ and its maximal rearrangement $F_u^{**}$. When $w_n$ is the volume of the unit ball, we have that
\begin{flalign*}
  \frac{F_u(B(x,R_k))}{R_k^{n-1}}&=\frac{1}{R_k^{n-1}}\int_{B(x,R_k)}F_u(y)\,dy\\&\leq {w_n R_k}\, \barint_0^{w_nR_k^n}F_u^*(t)\,dt={w_n R_k}\, F_u^{**}({w_nR_k^n})\,.
\end{flalign*}
Then  we have \begin{flalign}\label{1na8} R_kg^{-1}\left(\frac{F_u(B(x,R_k))}{R_k^{n-1}}\right)&\lesssim R_kg^{-1}\left({w_n R_k}\, F_u^{**}({w_nR_k^n})\right)\\ &\lesssim  \int_{w_n R_k^{n}}^{w_n R_{k-1}^n}\rho^\frac{1}{n}g^{-1}\left(\rho^\frac{1}{n} \, F_u^{**}(\rho)\right)\,\frac{d\rho}{\rho}\nonumber
\end{flalign}
with implicit constants independent of $k$. Therefore
\begin{flalign*}\sup_{x\in\Omega_0}\cW_G^{F_u}(x,R)&\lesssim \sum_{k=1}^\infty \int_{w_n R_k^{n}}^{w_n R_{k-1}^n}{ \rho^{\frac{1}{n}}} g^{-1}\left(\rho^\frac{1}{n}\, F_u^{**}(\rho)\right)\,\frac{d\rho}{\rho}\\
&= \int_{0}^{w_n R^n} \rho^\frac{1}{n} g^{-1}\left(\rho^\frac{1}{n}\, F_u^{**}(\rho)\right)\,\frac{d\rho}{\rho}\,.
\end{flalign*}
By the assumption on $F_u$, we get the  convergence of $\cW_G^{\mu_u}(x,R)$ to zero uniform with respect to $x\in\Omega_0$ as $R\to 0$. Then Corollary~\ref{coro:cont} gives the desired continuity of $u$ in every point of $\Omega_0$.
\end{proof}

By Corollary~\ref{coro:cont-2} we retrieve the following classical result.
\begin{rem}\label{rem:p-Lorentz}\rm If $u$ is a nonnegative $\opA$-superharmonic function satisfying $-\dv \opA(x,Du)=-{ \dv}\left(a(x)|D u|^{p-2} D u \right)=F_u$ in the sense of distributions for $p>1$  and $0<a\in L^\infty(\Omega)$  separated from zero and $F_u$ belongs locally to the  Lorentz space  {$L(\tfrac{n}{p},\tfrac{1}{p-1})(\Omega)$}
then $u$ is continuous. \end{rem}

We can extend the above remark to the Zygmund case.
\begin{rem} \rm Suppose that $1<p<n,$ $\alpha\in\R$, $0<a\in L^\infty(\Omega)$  separated from zero, and $u$ is a nonnegative $\opA$-superharmonic function in $\Omega$, such that  
\begin{equation}
    \label{eq-log-Fu}-\dv \opA(x,Du)=-{ \dv}\left(a(x)|D u|^{p-2}\log^\alpha({\rm e}+|D u|)  D u \right)=F_u
\end{equation} is satisfied in the sense of distributions. Observe that in this case $g^{-1}(\lambda)\simeq \lambda^\frac{1}{p-1}\log^{-\frac{\alpha}{p-1}}({\rm e}+\lambda)$. If $F_u$ satisfies~\eqref{cond-Lor}, then $u$ is continuous.
\end{rem} 
 
A density condition for a measure expressed in the relevant generalized Morrey-type scale is equivalent to H\"older continuity of the solution to measure data problem. Namely, we consider the class of measures for which there exist positive constants $c=c(\data,n)>0$ and $\theta\in (0,1)$ such that  \begin{equation}
    \label{mu-control}
\mu_{u,\theta}(B(x,r))\leq c r^{n-1}g(r^{\theta-1})\simeq r^{n-\theta}G(r^{\theta-1})
\end{equation}
for each $B(x,r)$ with $B(x,2r)\Subset\Omega$ and $r<\min\{1,R_{\mathcal{W}}/2\}$. This condition is a~natural Orlicz version of the related one from~\cite{Car,KiMa92,kizo,RaZi} and the one used to characterize removable sets for H\"older continuous solutions, cf.~\cite{ChKa}. 
Notice that for $p$-growth problems~\eqref{mu-control} reads as $\mu_{u,\theta}(B(x,r))\leq c  r^{n-p+\theta(p-1)}.$

\begin{coro} \label{coro:H-cont}  Under {\rm Assumption {\bf (A)}} suppose $u$ is a nonnegative $\opA$-superharmonic function in $\Omega$ and $\mu_u:=-\dv \opA(x,Du)$  in the sense of distributions. Assume further that $u\in C^{0,\theta}_{loc}(\Omega)$ with certain $\theta\in(0,1)$, then there exists a constant $c=c(\data,n)>0$, such that $\mu_u=\mu_{u,\theta}$ satisfies the density condition~\eqref{mu-control}.
\end{coro}{}
\begin{proof} 
Set $x\in B(x,r)\subset B(x,2r)\Subset\Omega$ with $r<\min\{1,R_{\mathcal{W}}/2\}$. We observe that $u-\inf_{B(x,2r)} u$ is $\opA$-superharmonic, so we can use lower estimate from Theorem~\ref{theo:est-sol} to conclude with
\begin{flalign*}
g^{-1}\left(\frac{\mu(x,r)}{r^{n-1}}\right)&\leq c\,\barint_r^{2r}g^{-1} \left(\frac{\mu(x,s)}{s^{n-1}}\right)\,ds\leq \frac{c}{r}\big( u(x)-\inf_{B(x,2r)} u+r\big)\leq c r^{\theta-1}
\end{flalign*} equivalent to~\eqref{mu-control}.
\end{proof}

\begin{coro}\label{coro:O-dens} Under {\rm Assumption {\bf (A)}} suppose that $u$ is a nonnegative $\opA$-superharmonic function finite a.e. in $B(x_0,4r)\Subset\Omega$ with $r<\min\{1,R_{\mathcal{W}}/2\}$  in the sense of distributions, and $\mu_{u,\theta}$ satisfies density condition~\eqref{mu-control} with certain $\theta\in(0,1)$. Then for some $c=c(\data,n)$ \[\sup_{B(x_0,r)}u\leq c\Big(\inf_{B(x_0,r)} u +r^\theta\Big)\]
and, consequently, $u$ is locally H\"older continuous.  
\end{coro}
\begin{proof} 
{Let us fix arbitrary $x\in B(x_0,r)$. Then one has $ B(x_0,r)\subset B(x,2r)\subset B(x_0,4r)$. 
By the upper estimate from Theorem~\ref{theo:est-sol} 
we get 
\begin{flalign*}
 u(x)&\leq C_U\left(\inf_{B(x,2r)}u+ 2r + \int_0^{2r} g^{-1}\left(c\frac{s^{n-1}g(s^{\theta-1})}{s^{n-1}}\right)\,ds\right)\\&
\leq c\left(\inf_{B(x_0,r)}u+ r^\theta\right).
\end{flalign*}
 We take supremum over $B(x_0,r)$ on the both sides of the inequality above and get the inequality from the claim. Further by classical iteration as in \cite[Chapter 6]{hekima} or \cite[Chapter 6]{giusti}, we get H\"older continuity of $u$.}
\end{proof}
Specializing Corollaries~\ref{coro:H-cont} and~\ref{coro:O-dens}, we have the following results.
\begin{rem} \rm Suppose $u$ is a nonnegative $\opA$-superharmonic function satisfying $-\dv \opA(x,Du)=-{ \dv}\left(a(x)|D u|^{p-2} D u \right)=\mu_u$ in the sense of distributions for $p>1$  and $0<a\in L^\infty(\Omega)$  separated from zero. Assume further that $u\in C^{0,\theta}_{loc}(\Omega)$ with certain $\theta\in(0,1)$, then $\mu_u$ satisfies  \begin{equation}
    \label{mu-p}
\mu_u(B(x,r))\leq c r^{n-p+\theta(p-1)}
\end{equation}
for some $c>0,$ $\theta\in(0,1)$ and all sufficiently small $r>0.$ On the other hand, if~\eqref{mu-p}, then $u$ is locally  H\"older continuous.
\end{rem}
\begin{rem} \rm Suppose $u$ is a nonnegative $\opA$-superharmonic function satisfying 
\begin{equation}
    \label{eq-log-mu}-\dv \opA(x,Du)=-{ \dv}\left(a(x)|D u|^{p-2}\log^\alpha({\rm e}+|D u|)  D u \right)=\mu_u
\end{equation}  in the sense of distributions for $p>1,$ $\alpha\in\R$, and $0<a\in L^\infty(\Omega)$  separated from zero. Assume further that $u\in C^{0,\theta}_{loc}(\Omega)$ with certain $\theta\in(0,1)$, then $\mu_u$ satisfies  \begin{equation}
    \label{mu-log}
\mu_u(B(x,r))\leq c r^{n-p+\theta(p-1)}\log^\alpha({\rm e}+r^{\theta-1})
\end{equation}
for some $c>0,$ $\theta\in(0,1)$ and all sufficiently small $r>0.$ On the other hand, if~\eqref{mu-log}, then $u$ is locally  H\"older continuous.
\end{rem}

The sufficient condition for~\eqref{mu-control} and, in turn, for the H\"older continuity of the solution is to assume that $\mu_u=F_u$ belongs to a~relevant Marcinkiewicz-type space $L(\psi,\infty)(\Omega)$. We say that $f\in L(\psi,\infty)(\Omega)$ if the maximal rearrangement $f^{**}$ of $f$ satisfies
\[\sup_{s\in(0,|\Omega|)}\frac{f^{**}(s)}{\psi^{-1}(1/s)}<\infty,\] 
see~\cite{oneil}.

\begin{coro}
\label{coro:O-Marc-dens}  Suppose $u$ is a nonnegative $\opA$-superharmonic function satisfying  $-\dv\opA(x,Du)=\mu_u=F_u$  in the sense of distributions, and $F_u$ belongs locally to the Marcinkiewicz-type space $L(\psi,\infty)(\Omega)$ with $\psi^{-1}(1/\lambda)=\lambda^{-\frac{1}{n}}g\big(\lambda^{\frac{\theta-1}{n}}\big)$ for some $\theta\in(0,1)$, 
then $u$ is locally H\"older continuous. \end{coro}
\begin{proof}
    By~\eqref{1na8} and the assumption we get that
\[rg^{-1}\left(\frac{F_u(B(x,r))}{r^{n-1}}\right) \leq c r^{\theta},\]
which is equivalent to~\eqref{mu-control}.    By Corollary~\ref{coro:O-dens} we get H\"older continuity of $u$.
\end{proof}
This fact has  the best possible consequence in the $p$-Laplace case.
\begin{rem}\rm \label{rem:p-Marc-dens} If $u$ is a nonnegative $\opA$-superharmonic function in $\Omega$ satisfying $-\dv \opA(x,Du)=-{ \dv}\left(a(x)|D u|^{p-2} D u \right)=F_u$ in the sense of distributions for $p>1$  and $0<a\in L^\infty(\Omega)$  separated from zero, and $F_u$ belongs locally to the Marcinkiewicz space $L(\frac{n}{p+\theta(p-1)},\infty)(\Omega)$  for some $\theta\in(0,1)$, i.e.  $\sup_{\lambda>0}\left(\lambda^\frac{n}{p+\theta(p-1)}\big|\{x\in\Omega_0:\,F_u(x)>\lambda\}|\right)<\infty$ for $\Omega_0\Subset\Omega$, then $u$ is locally H\"older continuous.
\end{rem}
 
\begin{rem}\label{rem:G-Marc-dens} \rm When $G(t)\simeq t^p\log^\alpha({\rm e}+t),$ $1<p<n,$ $\alpha\in\R,$ and $u$ is a~nonnegative $\opA$-superharmonic function satisfying \eqref{eq-log-Fu} and such that \[\sup_{\lambda>0}\left(\lambda^\frac{n}{p+\theta(p-1)}\log^{-\frac{\alpha(1-\theta)}{p+\theta(p-1)}}({\rm e}+\lambda^\frac{n}{1-\theta})\big|\{x\in\Omega_0:\,F_u(x)>\lambda\}|\right)<\infty\] for $\Omega_0\Subset\Omega$, then $u$ is locally H\"older continuous.
\end{rem}

\subsection{Hedberg--Wolff Theorem}\label{ssec:WH}  We present the general growth version of the theorem by Hedberg and Wolff. Let us refer to~\cite{HedWol} for the classical formulation and proof, later proof in~\cite{adams-hedberg}, and the variable exponent version elaborated in~\cite{LuMaMa}. Our proof applies the estimates from Theorem~\ref{theo:est-sol} and is given in Section~\ref{sec:W-H}.

\begin{theo}\label{theo:W-H}
Let  $\mu$ be a nonnegative bounded Radon measure compactly supported in bounded open set  $\Omega\subset\rn$. Then \begin{equation}\label{mu-in-dual} \mu\in (W^{1,G}_0(\Omega))' \end{equation} if and only if
\begin{equation} \label{potential-bounded} \int_\Omega \mathcal{W}^\mu_G(x,R)\,d\mu(x)<\infty\quad\text{for some }\ R>0.\end{equation}
\end{theo} 
\section{Preliminaries}\label{sec:prelim}

\subsection{Notation}
In the following we shall adopt the customary convention of denoting by $c$ a constant that may vary from line to line. Sometimes to skip rewriting a constant, we use $\lesssim$. By $a\simeq b$, we mean $a\lesssim b$ and $b\lesssim a$. By $B_R$ we shall denote a ball usually skipping prescribing its center, when {it} is not important. Then by $cB_R=B_{cR}$ we mean then a ball with the same center as $B_R$, but with rescaled radius $cR$. We make use of symmetric truncation on level $k>0$,  $T_k:\R\to\R$, defined as follows 
\begin{equation}\label{Tk}T_k(s)=\left\{\begin{array}{ll}s & |s|\leq k,\\
k\frac{s}{|s|}& |s|\geq k.
\end{array}\right. 
\end{equation}
 With $U\subset \mathbb{R}^{n}$ being a~measurable set with finite and positive $n$-dimensional Lebesgue measure $|U|>0$, and with $f\colon U\to \mathbb{R}^{k}$, $k\ge 1$ being a measurable map, by
\begin{flalign*}
\barint_{U}f(x) \, dx =\frac{1}{|U|}\int_{U}f(x) \,dx
\end{flalign*}
we mean the integral average of $f$ over $U$. 
By $C^{0,\gamma}(U)$, $\gamma \in (0,1]$, we mean the family of H\"older continuous functions, i.e. those measurable functions $f\colon U\to \mathbb{R}$ for which
\begin{flalign*}
[f]_{0,\gamma}:=\sup_{\substack{x,y\in U,\\x\not =y}}\frac{|f(x)-f(y)|}{|{x-y}|^{\gamma}}<\infty.
\end{flalign*} 

\subsection{Basic definitions} References for this section {are~\cite{rao-ren,KR}}. 
 
We say that a function $G: [0, \infty) \to [0, \infty)$ is an $N$-function if it is convex, vanishes only at $0$, and satisfies the
additional growth conditions
\begin{equation*}\lim _{t \to
0}\frac{G (t)}{t}=0 \qquad \hbox{and} \qquad \lim _{t \to \infty
}\frac{G (t)}{t}=\infty \,.
\end{equation*} 

 The  complementary~function $\wt{G}$  (called also the Young conjugate, or the Legendre transform) to a nondecreasing function $G:\rp\to\rp$  is given by the following formula
\[\wt{G}(s):=\sup_{t>0}(s\cdot t-G(t)).\]

In the general growth case Young's inequality reads as inequality\begin{equation}
\label{in:Young} ts\leq G(t)+\wt{G}(s)\quad\text{for all }\ s,t\geq 0.
\end{equation}

 We say that a function $G:\rp\to\rp$ satisfies $\Delta_2$-condition if there exist $c_{\Delta_2},t_0>0$ such that $G(2t)\leq c_{\Delta_2}G(t)$ for $t>t_0.$ It describes the speed and {the} regularity of the growth. We say that $G$ satisfy $\nabla_2$-condition if $\wt{G}\in\Delta_2.$ 
Note that it is possible that $G$ satisfies only one of the conditions $\Delta_2/\nabla_2$. For instance, when $G(t) = (1+|t|)\log(1+|t|)-|t|$, its complementary function is  $\widetilde{G}(s)= \exp(|s|)-|s|-1$. Then $G\in\Delta_2$ and   $\widetilde{G}\not\in\Delta_2$.

See~\cite[Section~2.3, Theorem~3]{rao-ren} for equivalence of various definitions of this condition. In particular, $G\in\Delta_2\cap\nabla_2$ if and only if
\begin{equation}\label{iG-sG} 
1<  i_G=\inf_{t>0}\frac{tg(t)}{G(t)}\leq \sup_{t>0}\frac{tg(t)}{G(t)}=s_G<\infty,\end{equation}
where $g(t)=G'(t)$. This assumption implies a comparison with power-type functions i.e.
\begin{equation}\label{comp-i_G-s_G} 
\frac{G(t)}{t^{i_G}}\quad\text{is non-decreasing}\qquad\text{and}\qquad\frac{G(t)}{t^{s_G}}\quad\text{is non-increasing}.
\end{equation} 
 
 \begin{lem}\label{lem:equivalences}
 If {an $N$-function} $G\in\Delta_2\cap\nabla_2$, then \begin{flalign*} g(t)t\simeq G(t)
 \qquad\text{and}\qquad \wt G(g(t))\simeq G(t)
 \end{flalign*}
 with the constants depending only on the growth indexes of $G$, that is $i_G$ and $s_G$.  Moreover,  $g^{-1}(2t)\leq c g^{-1}(t)$ with $c=c(i_G,s_G).$
 \end{lem}
 \subsection{Orlicz spaces}
 Basic reference  for this section is~\cite{adams-fournier}. 
 
We study the solutions to PDEs in the Orlicz-Sobolev spaces equipped with a modular function $G\in C^1 {((0,\infty))}$ - a strictly increasing and convex function   such that $G(0)=0$ and satisfying~\eqref{iG-sG}. Let us define a modular \begin{equation}
    \label{modular}
\vr_{G,U}(u)=\int_U G(|u|)\,dx.
\end{equation}
 
\begin{defi}\label{def:OrSob:sp} 
For any bounded $\Omega\subset\rn$, by Orlicz space ${L}^G(\Omega)$  we understand the space of measurable functions endowed with the Luxemburg norm 
\[||f||_{L^G(\Omega)}=\inf\left\{\lambda>0:\ \ \vr_{G,\Omega}\left( \tfrac{1}{\lambda} |f|\right)\leq 1\right\}.\]
 We define the Orlicz-Sobolev space  $W^{1,G}(\Omega)$  as follows
\begin{equation*} 
W^{1,G}(\Omega)=\big\{f\in W^{1,1}_{loc}(\Omega):\ \ |f|,|D f|\in L^G(\Omega)\big\},
\end{equation*}where the gradient is understood in the distributional sense, endowed with the norm
\[
\|f\|_{W^{1,G}(\Omega)}=\inf\bigg\{\lambda>0 :\ \    \vr_{G,\Omega}\left( \tfrac{1}{\lambda} |f|\right)+ \vr_{G,\Omega}\left( \tfrac{1}{\lambda} |Df|\right)\leq 1\bigg\} 
\]
and  by $W_0^{1,G}(\Omega)$ we denote the closure of $C_0^\infty(\Omega)$ under the above norm. 
\end{defi} 

\begin{lem}\label{lem:doubling-norm}For any $N$-function $G$ we have that $\|f\|_{L^G(\Omega)}\leq \vr_{G,\Omega}(|f|)+1$. If additionally $G\in\Delta_2\cap\nabla_2$, then $\vr_{G,\Omega}\left(|f|\right)$ is bounded if and only if $f\in {L^G(\Omega)}$.
\end{lem}

The counterpart of the H\"older inequality in this setting reads \begin{equation}
\label{in:Hold} \|fg\|_{L^1(\Omega)}\leq 2\|f\|_{L^G(\Omega)}\|g\|_{L^{\wt{G}}(\Omega)}\quad\text{for all }\ f\in L^G(\Omega),\ g\in L^{\wt{G}}(\Omega).
\end{equation}

\begin{rem}\rm  \cite{adams-fournier} Since condition~\eqref{iG-sG} imposed on $G$ implies $G,\wt{G}\in\Delta_2$, the Orlicz-Sobolev space $W^{1,G}(\Omega)$ we deal with is separable and reflexive.
\end{rem}

\begin{prop}[Modular Poincar\'e inequality, \cite{Baroni-Riesz,MSZ}]  \label{prop:Poincare} 
For {an $N$-function} $G\in {C^1((0,\infty))}$, such that $G\in \Delta_2 \cap \nabla_2 $, there exists a constant $c=c(n,i_G,s_G)$, such that
\[\int_{B_R} G\left(\frac{|f|}{R}\right)\,dx\leq c \int_{B_R} G\left( |Df|\right)\,dx\quad\text{for every $f\in W^{1,G}_0(B_R).$}\]
\end{prop} 
{ For other modular Poincar\'e inequalities in the framework of Orlicz spaces see e.g \cite[Example  3.1]{FioGia}, \cite{GP}, \cite[Lemma 2.2]{lieb}. }
\begin{rem}\label{rem:gen-grad}\rm
If $T_k u \in W^{1,G}(\Omega)$ for every $k>0$, then there exists a unique measurable function $Z_u:\Omega\to\rn$ such that $D (T_{{k}}(u))=\mathds{1}_{\{|u|<{{k}}\}}Z_u$ a.e. in $\Omega$,  for every ${k},$
see \cite[Lemma~2.1]{bbggpv}. For $u\in W^{1,G}(\Omega)$, we have $Z_u=D u$ a.e. in $\Omega$. Thus,  we call $Z_u$ the generalized gradient of $u$ and, abusing the notation, for $u$ such that $T_k u \in W^{1,G}(\Omega)$ for every $k>0$ we write simply $D u$ instead of $Z_u$. In all cases this notation means $Du=\lim_{k\to\infty}D(T_k u).$
\end{rem}

\subsection {The operator} We notice that in such regime  the operator   $\mathfrak{A}_{G}$ acting as
\begin{flalign*}
\langle\mathfrak{A}_{G}u,\phi\rangle:=\int_{\Omega}\opA(x,Du)\cdot D\phi \, dx\quad \text{for}\quad \phi\in C^{\infty}_{0}(\Omega)
\end{flalign*}
is well defined on a reflexive and separable Banach space $W^{1,G}(\Omega)$ and $\mathfrak{A}_{G}(W^{1,G}(\Omega))\subset (W^{1,G}(\Omega))'$. Indeed, when $u\in W^{1,G}(\Omega)$ and $\phi\in C_0^\infty(\Omega)$, growth conditions~\eqref{ass-op},  H\"older's inequality~\eqref{in:Hold}, and Lemma~\ref{lem:equivalences}  justify that
\begin{flalign*}
|{\langle \mathfrak{A}_{G}u,\phi \rangle}|\le &\,c\int_{\Omega}\frac{G(|{Du}|)}{|{Du}|}{|D\phi|} \, dx \le c\left \| \frac{G(|{Du}|)}{|{Du}|}\right \|_{L^{\wt G(\cdot)}(\Omega)}\|{D\phi}\|_{L^{G}(\Omega)}\nonumber \\
\le &\, c\|{Du}\|_{L^{ G}(\Omega)}\|{D\phi}\|_{L^{G}(\Omega)}\le c\|{\phi}\|_{W^{1,G}(\Omega)}.
\end{flalign*}

\subsection{Solutions, approximable solutions, $\opA$-supersolutions and $\opA$-harmonic functions}
\label{ssec:sols} 

All the problems are considered under {\rm Assumption {\bf (A)}}, see Section~\ref{sec:formulation-n-reg-cons}. 
A \underline{continuous} function $u\in W^{1,G}_{loc}(\Omega)$ is an {\em $\opA$-harmonic} function in an open set $\Omega$ if it is a (weak) solution to the equation $
-\dv\opA(x,Du)= 0$, i.e.
\[\int_\Omega \opA(x,Du)\cdot D\phi\,dx=0\quad\text{for every }\ \phi\in C^\infty_0(\Omega).\]
As a direct consequence of \cite[Theorem~2]{ChKa} we infer the following. 
\begin{prop}[Existence of $\opA$-harmonic functions] \label{prop:ex-Ahf} Under {\rm Assumption {\bf (A)}} if $\Omega$ is bounded and $w\in W^{1,G}(\Omega)\cap C(\Omega)$, then there exists a unique solution $u\in W^{1,G}(\Omega)\cap C(\Omega)$ to problem
\begin{equation*}
\begin{cases}-\dv\, \opA(x,Du)= 0\quad\text{in }\ \Omega,\\
u-w\in W_0^{1,G}(\Omega).\end{cases}
\end{equation*}
Moreover, for every $E\Subset\Omega$ we have $\quad\|u\|_{L^\infty(E)}\leq c(\data, \|Du\|_{L^{G}(\Omega)}).$
\end{prop}
 
We call a function $u\in W^{1,G}_{loc}(\Omega)$ a (weak) {\em $\opA$-supersolution} to~\eqref{eq:0} if~$-\dv\opA(x,Du)\geq 0$ weakly in $\Omega$, that is 
\begin{equation*}
\int_\Omega \opA(x,Du)\cdot D\phi\,dx\geq 0\quad\text{for all }\ 0\leq\phi\in C^\infty_0(\Omega)
\end{equation*}
and a (weak) {\em $\opA$-subsolution} if $-\dv\opA(x,Du)\leq 0$ weakly in $\Omega$, that is \begin{equation}
\label{eq:sub}
\int_\Omega \opA(x,Du)\cdot D\phi\,dx\leq 0\quad\text{for all }\ 0\leq\phi\in C^\infty_0(\Omega).
\end{equation}
We consider a measure data problem \begin{equation}
    \label{eq:mu}\begin{cases}-\dv \opA(x,Du)=\mu\quad\text{in }\ \Omega,\\
    u=0\quad\text{on }\ \partial\Omega.\end{cases}
\end{equation}
A function $u\in W^{1,G}_{loc}(\Omega)$ is called {a} {\em weak solution} to~\eqref{eq:mu}, if\begin{equation}
    \label{eq:main-mu-weak}\int_\Omega \opA(x,Du)\cdot D\phi\,dx=\int_\Omega\phi\,d\mu(x)\quad\text{for every }\ \phi\in C^\infty_0(\Omega).
\end{equation}
Recall that $W^{1,G}_{0}(\Omega)$ is separable and by its very definition ${C_0^\infty}(\Omega)$ is dense there.

\begin{rem}[Existence and uniqueness of weak solutions] \rm For $\mu\in (W^{1,G}_{0}(\Omega))',$ due to the strict monotonicity of the operator, there exists a unique weak solution to~\eqref{eq:mu}, see~\cite{KiSt}.
\label{rem:weak-sol}
\end{rem}

Obviously, for arbitrary 
measure one cannot expect {that weak solutions to~\eqref{eq:mu}  exist}, but there is a suitable notion of solutions we can employ here. They are called `approximable', built like SOLA but involving the use of truncations~\eqref{Tk}. We define them as in~\cite{CiMa}. 
A function $u$, such that $T_s u\in W_{loc}^{1,G}(\Omega)$ for every $s>0$,  is called an `{\em approximable solution}' to the Dirichlet problem~\eqref{eq:mu} with a given bounded Radon measure $\mu$, if there exists a sequence $\{f_k \}_k\subset L^1 (\Omega) $ converging weakly-$\ast$ to $\mu$ in the space of measures, i.e. {such that}
\[
\lim_{k\to\infty}\int_\Omega\phi\,f_k\,dx=\int_\Omega\phi\,d\mu
\quad\text{for every $\phi\in C_0(\Omega)$}\]
 and a sequence of~weak solutions $\{u_k\}_k\subset W^{1,G}_0 (\Omega)$ to problem~\eqref{eq:mu} with $\mu$ replaced by $f_k$, satisfying $u_k\to u$  {a.e. in }$\Omega$.

\begin{rem}[Existence and uniqueness of `approximable solutions']\label{rem:approx-sols-exist}\rm Due to~\cite{CiMa} an `approximable solution' exists for every bounded Radon measure $\mu$ and then $\opA(x,Du_k)\to \opA(x,Du)$ {a.e. in }$\Omega$ with generalized gradient, cf. Remark~\ref{rem:gen-grad}. When the datum is absolutely continuous with respect to Lebesgue's measure the solutions not only exist, but they are also proven to be unique.
\end{rem}

Let us additionally comment that the use of truncations in the definition of `approximable solutions' make us independent of typical restrictions for the growth of the operator from below (we do not need to assume $p>2-\frac{1}{n}$). Indeed, a~very weak solution obtained as a limit of approximation may not have a locally integrable distributional gradient, but its trucation is $W^{1,1}_{loc}$-regular.

The classes of {\em $\opA$-superharmonic} and {\em $\opA$-subharmonic} functions are defined by the Comparison Principle.
\begin{defi}\rm
{A} lower semicontinuous function $u$ is {said to be} $\opA$-superharmonic if for any $K\Subset\Omega$ and any $\opA$-harmonic $h\in C(\overline {K})$ in $K$, $u\geq h$ on $\partial K$ implies $u\geq h$ in $K$. We say that an upper semicontinuous function $u$ is $\opA$-subharmonic if $(-u)$ is $\opA$-superharmonic.
\end{defi}

\medskip
Note that we do not assume for the definition that $\opA$-superharmonic function is finite a.e., though it is assumed in the hypotheses of our main theorems.

Directly from the definition we see that functions $\min\{u,v\}$ and $a_1u+a_2$ are $\opA$-superharmonic provided $u$ and $v$ are and $a_1,a_2\in\R$, $a_1\geq 0.$ 

We have the following relations between $\opA$-superharmonic functions and $\opA$-supersolutions.
\begin{lem}[Lemma 4.4, \cite{CZG-gOp}]\label{lem:cont-supersol-are-superharm} If $u$ is a continuous $\opA$-supersolution, then it is $\opA$-superharmonic.
\end{lem}
\begin{lem}[Lemma 4.6, \cite{CZG-gOp}]\label{lem:loc-bdd-superharm-are-supersol}
If $u$ is $\opA$-superharmonic in $\Omega$ and locally bounded from above, then $u\in W^{1,G}_{loc}(\Omega)$ and $u$ is  $\opA$-supersolution in $\Omega$.
\end{lem}

Note that within our regime $g=G'$ is strictly increasing, but not necessarily convex. Although in general $g$ does not generate the Orlicz space and does not support Poincar\'e inequality, we still can define modular $\vr_{g,\Omega}$ as in~\eqref{modular} and by its means describe fine properties of $\opA$-harmonic functions.  When $G$ is growing slowly, one cannot expect uniform integrability of gradients of truncations of solutions, but we can substitute it with the following result. 

\begin{rem}\rm \label{rem:g-intergrability-of-Ah}  If $u$ is an $\opA$-superharmonic function, which is finite a.e. in $\Omega$, we can cover any  set compactly included in $\Omega$ with finite number of balls $B$ of equal radius and such that $\vr_{g,B}(|Du|)\leq 1$.\\
Indeed, due to \cite[Remark~4.13]{CZG-gOp} a generalized gradient of $u$ is well-defined. Moreover, by \cite[Lemma~4.5]{CiMa} or \cite[Lemma~4.12]{CZG-gOp} there exists a function $\zeta_{\rm grad}:[0,|\Omega|]\to\rp$, such that $\lim_{s\to 0^+}\zeta_{\rm grad}(s)=0$ and for every measurable set $E\subset B$ it holds that $\vr_{g,E}(|Du|)\leq \zeta_{\rm grad}(|E|).$
\end{rem}
 
\subsection{$\opA$-superharmonic functions generate measures}\label{ssec:A-sh-gen-meas} 
 $\opA$-supersolutions can be characterized as solutions to measure data problems. This follows from the fact that a nonnegative distribution is a nonnegative measure (see e.g.  \cite[Th\'eor\`eme V]{s-book}). As a~consequence, we can state what follows.
\begin{lem}\label{lem:mu-great}
If a function $u\in W^{1,G}_{loc}(\Omega)$ is an $\opA$-supersolution to~\eqref{eq:0}, then there exits a nonnegative measure $ \mu_u\in (W^{1,G}_0(\Omega))'$, such that~\eqref{eq:mu} holds with $\mu_u$ being a nonnegative Radon measure, that is
\[
\int_\Omega \opA(x,Du)\cdot D\phi\,dx= \int_\Omega \phi\,d\mu_u(x)\quad\text{for all }\ 0\,\leq \, \phi\in C^\infty_0(\Omega).\]
\end{lem}

With $\opA$-superharmonic functions the situation is a bit more subtle. In general an~$\opA$-superharmonic function does not have to belong to $W^{1,1}_{loc}(\Omega),$ it is only the~limit of $\{T_k u\}\subset W^{1,G}_{loc}(\Omega)$ being $\opA$-supersolutions (see Lemma~\ref{lem:loc-bdd-superharm-are-supersol}). Nonetheless by~\cite[Remark 4.13]{CZG-gOp} generalized gradient of $u$, in the sense of Remark~\ref{rem:gen-grad}, is well-defined. Henceforth,  we have the following observation.
\begin{prop}\label{prop:A-sh-gen-meas} Suppose that $u$ is $\opA$-superharmonic and finite a.e. in $\Omega,$ then there is a nonnegative Radon measure $\mu_u$ on $\Omega,$ such that \begin{equation*}
   \int_\Omega \opA(x,Du)\cdot D\phi\,dx=\int_\Omega\phi\,d\mu_u(x)\quad\text{for every }\ \phi\in C^\infty_0(\Omega).
\end{equation*}
\end{prop}
Recall that $\mu_u$ as a nonnegative distribution is a nonnegative measure.

\subsection{Auxiliary results from Orlicz potential theory}\label{ssec:aux}
In this section we present results of~\cite{CZG-gOp}  specified to  our case.

 
\begin{lem}[Comparison Principle, Lemma 3.5, \cite{CZG-gOp}]\label{lem:cp} Let $u\in W_{loc}^{1,G}(\Omega)$ be an $\opA$-supersolution, and $v\in W^{1,G}_{loc}(\Omega)$ be an $\opA$-subsolution. If $\min(u- {v}) \in W^{1,G}_0(\Omega)$, then $u \geq  {v}$  a.e. in $\Omega$.
\end{lem}
As a direct consequence of comparison with $v\equiv 0$ solving $-\dv\opA(x,Dv)=0$, we have the following conclusion.
\begin{rem}\label{rem:nonnegativeness}  \rm
$\opA$-supersolutions are nonnegative a.e. 
\end{rem}

\begin{prop}[Harnack's Principle, Theorem 2,~\cite{CZG-gOp}]
\label{prop:har-princ}  Suppose that $u_i$, $i=1,2,\ldots$, are $\opA$-superharmonic and finite a.e. in $\Omega$. If the sequence $\{u_i\}$ is nondecreasing then the limit function $u=\lim_{i \to \infty} u_i$ is  $\opA$-superharmonic or infinite in~$\Omega$.  Furthermore, if $u_i$, $i=1,2,\ldots$, are nonnegative, then up to a subsequence also $Du_i\to Du$ a.e. in $\{u<\infty\},$ where `$D$' stands for the generalized gradient.
\end{prop}{}

\begin{prop}[Minimum Principle, Theorem 4,~\cite{CZG-gOp}]\label{prop:mini-princ} Suppose $u$ is $\opA$-superharmonic and finite a.e.  in $\Omega$. If $E\Subset\Omega$ a connected open subset of $\Omega$, then
\[\inf_{E} u=\inf_{\partial E} u.\]
\end{prop}

\begin{prop}[{Maximum Principle}, Corollary 4.16,~\cite{CZG-gOp}]\label{prop:max-princ} Suppose $u$ is $\opA$-subharmonic and finite a.e. in $\Omega$. If $E\Subset\Omega$ a connected open subset of $\Omega$, then
\[\sup_E u=\sup_{\partial E} u.\]
\end{prop}

Recall that by Remark~\ref{rem:g-intergrability-of-Ah}, we can cover a compact set included in our domain with finite number of balls $B$ of equal radius and such that $\vr_{g,B}({|}Du{|})\leq 1$, so without any loss of generality we can state a favorable Harnack's inequality over such small balls.
\begin{prop}[Harnack's inequality, Theorem 1,~\cite{CZG-gOp}]
\label{prop:harnack} Suppose  $u$ is an $\opA$-harmonic  and nonnegative function in a connected set $\Omega$. Then there exist  $R_H=R_H(n)>0$ and $C=C(\data,n,R_H,{\rm ess\,sup}_{B_{R_H}}u)>0
$ such that  
\begin{equation}\label{in:harnack-internal}\sup_{ B_R} u\leq C(\inf_{ B_R} u+R)\end{equation} 
for all $R\in(0,R_H]$ provided $B_{3R}\Subset\Omega$ and $\vr_{g,B_{3R}}(|Du|)\leq 1$.
\end{prop}  
\noindent Inequality~\eqref{in:harnack-internal} was provided first for solutions to problems with Orlicz growth with a constant dependent on $\sup_{B_R} u$, see~\cite{lieb}. For superquasiminimizers a similar result is given in~\cite{hh-zaa}, but the proof for Orlicz $\opA$-harmonic functions was not proven before~\cite{CZG-gOp}.
\begin{prop}[Corollary 4.17, \cite{CZG-gOp}]\label{prop:int-sph-harn}
Suppose $u$ is $\opA$-harmonic in $B_{\frac{3}{2}R}\setminus B_R,$ then there exist $R_H,C>0$ from Proposition~\ref{prop:harnack}, such that
\begin{equation*}
\sup_{\partial B_{\frac 43 R}} u\leq C(\inf_{\partial B_{\frac 43 R}} u+2R) \end{equation*}
for all $R\in(0,R_H]$ provided $B_{3R}\Subset\Omega$ and $\vr_{g,B_{3R}}(|Du|)\leq 1$.\end{prop}{}

Since an $\opA$-supersolution is always a superminimizer (see \cite[Lemma 3.7]{CZG-gOp}), we can specify~\cite[Theorem 4.3]{hh-zaa} in the following way. 
\begin{prop}[Weak Harnack estimate for $\opA$-supersolutions]\label{prop:inf-est}  Suppose $u \in W_{loc}^{1,G} (\Omega)$ is a nonnegative $\opA$-supersolution. Then for $R_H(n)>0$, $s_0=s_0(\data,n),$ and $c=c(\data,n)>0$, such that \[ \left(\barint_{B_{2 R}}u^{s_0}\,dx\right)^\frac{1}{s_0} \leq c \left({\rm ess\,inf}_{B_R} u +R\right),\]for all $R\in(0,R_H]$ provided $B_{3R}\Subset\Omega$ and $\vr_{g,B_{3R}}(|Du|)\leq 1$.\end{prop}
Let us note that after we completed our manuscript, the weak Harnack inequalities with an explicit exponent for unbounded supersolutions was proven in~\cite{bhhk}.

\subsection{Properties of the Poisson modification}\label{ssec:Pois}
The {\em Poisson modification} of an $\opA$-superharmonic function in a regular set $E$ carries the idea of local smoothing of $\opA$-superharmonic functions. A boundary point is called regular if at this point the boundary value of any Orlicz-Sobolev function is attained not only in the Sobolev sense but also pointwise.  See~\cite{hh-zaa} for the  result (in the generalized Orlicz case) that if the complement of $\Omega$ is locally fat at $x_0\in\partial\Omega$ in the capacity sense, then $x_0$ is regular. A set is called regular if all of its boundary points are regular. In particular, polyhedra, balls $B_R$ and annuli $B_{R}\setminus B_{\ve R},$ $\ve\in (0,1)$ are regular.

Let us consider a function $u$, which is $\opA$-superharmonic and finite a.e. in $\Omega$ and a  open set $E\Subset\Omega$ with regular $\overline{E}$. We define
\[u_E=\inf\{v:\ v \ \text{is $\opA$-superharmonic in $E$ and }\liminf_{y\to x}v(y)\geq u(x)\ \text{for each }x\in\partial \overline{E}\}\]
and the {\em Poisson modification} of $u$ in $E$ by
\begin{equation}
    \label{PuD}
P(u,E)=\begin{cases}
u\quad&\text{in }\ \Omega\setminus E,\\
u_E &\text{in }\  E.
\end{cases}{}
\end{equation}{}
The Poisson modification has the following nice properties.
\begin{prop}[Theorem 3,~\cite{CZG-gOp}]\label{prop:Pois} If $u$ is $\opA$-superharmonic {and finite a.e.} in $\Omega$, then its Poisson modification $P(u,E)$ is  $\opA$-superharmonic in~$\Omega$,  $\opA$-harmonic in $E$, and $P(u,E)\leq u$ in~$\Omega.$
\end{prop}{}

\subsection{Caccioppoli-type estimate}


\begin{prop}\label{prop:caccI} Suppose $v \in W_0^{1,G} (\Omega)$ is a nonnegative $\opA$-subsolution and a~cutoff function $\eta\in C_0^\infty(B_{2R})$ is such that $\mathds{1}_{B_R}\leq\eta\leq \mathds{1}_{B_{2R}}$ and $|D\eta|\leq c/R$. For any $q\geq s_G$ there exists $c>0$, such that  
\begin{equation}
\label{inq:caccI} 
\int_{B_{ 2R}}G(|Dv|)\eta^q\,dx\leq c \int_{B_{2R}}G\left({v}|D\eta|\right)
\,dx.
\end{equation}
\end{prop}
\begin{proof} We test~\eqref{eq:sub} with $\xi=\eta^q v$ to get
\[\int_\Omega \opA(x,Dv)\cdot Dv\,\eta^q \,dx \leq-q\int_\Omega \opA(x,Dv)\cdot D\eta\, \eta^{q-1}v \,dx
.\]
Therefore, due to coercivity of $\opA$ and the Cauchy-Schwartz inequality we have
\begin{equation*}\int_{B_{2R}}G(|Dv|)\eta^q\,dx\leq c\int_{B_{2R}}g(|Dv|)|D\eta|\eta^{q-1}{v}
\,dx=:\mathcal{K}
\end{equation*}
Noting that $q$ is large enough to satisfy $s_G'\geq q'$, we have in turn that $\wt{ G}(\eta^{q-1}t)\leq c \eta^q\wt{G}(t)$  and  \[\wt{G}(\eta^{q-1}g(t))\leq c \eta^q \wt{G}(g(t))\leq c \eta^q G(t).\] 
Then,  using Young inequality~\eqref{in:Young} applied to the integrand of $\mathcal{K}$ we get
\begin{equation*}
\begin{split}\mathcal{K}&\leq \ve \int_{B_{2R}}\wt{G}(\eta^{q-1}|Dv|)\, dx+ c_\ve \int_{B_{2R}}G\left({v}|D\eta|\right)\,
dx\\& \leq \ve c\int_{B_{2R}} \eta^{q }G(|Dv| ) \,dx+ c_\ve \int_{B_{2R}}G\left({v}|D\eta|\right)\,
dx
\end{split}
\end{equation*}
with arbitrary $\ve<1$. Choosing $\ve$ small enough to absorb the term, and noticing that $\mathds{1}_{B_R}\leq\eta\leq \mathds{1}_{B_{2R}},$ we obtain~\eqref{inq:caccI}.
\end{proof}

\section{Main proof} \label{sec:mainproof}
The organization of this section is as follows. Subsection~\ref{ssec:reductions} provides a bunch of remarks on our proof. In Subsection~\ref{ssec:lower-bound} we prove the lower bound, while in Subsection~\ref{ssec:upper-bound}
 the upper bound.

\subsection{Reductions and remarks on the proof}\label{ssec:reductions}
 Our main steps follow the scheme of~\cite{KoKu,tru-wa} and essentially employ nonstandard growth potential theory tools coming from recent paper~\cite{CZG-gOp} presented in Section~\ref{ssec:aux}. 

We justify here that without loss of generality of the result itself, we can significantly simplify our proof. Namely, it is enough to prove Theorem~\ref{theo:est-sol} for continuous $\opA$-supersolutions. We recall that Section~\ref{ssec:A-sh-gen-meas} explains how an $\opA$-superharmonic function generates a measure $\mu_u$.

 To motivate that in our proof $u$ can be assumed to be continuous let us remark that it is  lower semicontinuous by the very  definition of an $\opA$-superharmonic function. To find approximation from below by continuous functions we proceed as in the proofs of~\cite[Proposition~4.5 and Lemma~4.6]{CZG-gOp}.  Let us consider a nondecreasing sequence $\{\phi_j\}_j$ of {nonnegative} Lipschitz functions converging pointwise to $u$. Considering the Dirichlet problems with obstacles $\phi_j,$ $j=1,2,\dots,$ and boundary datum $u$, we get the sequence $\{u_j\}_j$ of nonnegative continuous $\opA$-supersolutions converging to $u$ pointwise with $Du_{j}\to Du$ a.e. for some non-relabelled subsequence and generalized gradient `$D$'. For well-posedness and basic properties of the obstacle problem see~\cite[Section~4]{ChKa} and~\cite{kale}. By Lemma~\ref{lem:cp}, the sequence  $\{u_j\}$ is nondecreasing. 
 Then for every $j$, we have that $\{T_k u_j\}_k$  is a nondecreasing sequence of continuous functions converging to $u_j$ and by 
 Lemma \ref{lem:mu-great} they generate a sequence of measures
 $\{\mu_{T_k u_j}\}_k\subset (W^{1,G}_0(\Omega))'$. 
 Note that $\{\mu_{T_k u_j}\}_k$  locally converge weakly-$*$ to $\mu_{u_j}$. Indeed,  $u_j$ is locally bounded and therefore, by Proposition~\ref{prop:caccI}, we infer that  $\{\vr_{G,\Omega'}(|D T_k u_j|)\}_k$ is uniformly bounded for every $\Omega'\Subset\Omega$. Then by Lemma~\ref{lem:doubling-norm}, we get that 
 $\{T_k u_j\}_k$ is locally uniformly bounded in $ W^{1,G}_0(\Omega)$  and, consequently, we may pass to its (non-relabelled) weakly convergent subsequence {in $ W^{1,G}_0(\Omega)$}.  Thus, when we fix arbitrary $\xi\in W^{1,G}_0(\Omega)$ with $\supp\,\xi\Subset\Omega$, then  reasoning as in \cite[Lemma~4.6]{CZG-gOp}, we have 
{$$\opA(\cdot, DT_k{u_j})\rightharpoonup{}\opA(x, D{u_j})\quad \text{weakly in}\quad (L^{\wt G}(\Omega))^n$$}
{and hence}
 \begin{flalign*}  \lim_{k\to\infty}\int_\Omega \xi \, d\mu_{T_k u_j}&=\lim_{k\to\infty}\int_\Omega \opA(x, D{T_k u_j})\cdot D\xi\,dx\\&=\int_\Omega \opA(x, D{u_j})\cdot D\xi\,dx  =\int_\Omega \xi\, d\mu_{u_j}.\end{flalign*}
 Choosing diagonally subsequence of $\{T_k u_j\}_{k,j}$, we get a nondecreasing sequence $\{u^i\}_i$ of continuous and bounded $\opA$-supersolutions converging pointwise to~$u$ and such that $Du^i\to Du$ a.e. in $\Omega$. Then the corresponding measures $\mu_{u^{i}}$  locally converge  weakly-$\ast$ to $\mu_{u}$ in the space of measures.


{\em Lower bound.} We note that
\[\mu_{u}(\overline{B(x_0,r)})\leq \liminf_{i\to\infty} \mu_{u^{i}}(\overline{B(x_0,r)}).\]
When we take $R$ as in the hypothesis and $\ve\ll R$ we have\[ 
\mathcal{W}^{\mu_u}_G(x_0,R)\leq \int_0^{R+\ve} g^{-1}\left(\frac{\mu_u(\overline{B(x_0,r)})}{r^{n-1}}\right)\,dr.\] 
By making use of the above facts, Fatou's lemma, and finally sending $\ve\to 0$ we get the upper estimate.

{\em Upper bound.} We use analogous arguments and the observation that
\[\limsup_{i\to\infty} \mu_{u^{i}}(\overline{B(x_0,r)})\leq  \mu_{u}(\overline{B(x_0,r)}).\]

\subsection{Proof of lower bound in Theorem~\ref{theo:est-sol}} \label{ssec:lower-bound}
{In Section~\ref{ssec:reductions} we motivate that it is enough to prove~\eqref{eq:u-est} for $u$ being continuous and bounded $\opA$-supersolutions.}
\begin{proof} Fix $R\in (0,R_{\cW}/2)$, then $B(x_0,2R)\Subset \Omega$. We set \begin{equation}
    \label{RkBk}
R_k=2^{1-k}R\quad\text{and}\quad B_k=B(x_0,R_k),\ \ k=0,1,\dots\,.
\end{equation}
Since $u$ is an $\opA$-supersolution, then by Lemma~\ref{lem:mu-great} there exists a~nonnegative measure $\mu_u\in (W^{1,G}_0(B_k))'$ such that\[-\dv \opA(x,D u)=\mu_u\geq 0.\]
Then having $\theta_k\in C_0^\infty(\tfrac{5}{4}B_{k+1})$ such that $\mathds{1}_{B_{k+1}}\leq \theta_k\leq \mathds{1}_{\frac{5}{4}B_{k+1}}$, we set
\begin{equation*}
\mu_{w_k}:=\theta_k\mu_u\quad\text{in }\ B_k.
\end{equation*}
Note that \begin{equation}
    \label{muu=muwk}\mu_{w_k}(B_{k+1})=\mu_u(B_{k+1}).
\end{equation} Moreover, we have $\mu_{w_k}\in (W^{1,G}_0(B_k))'$. Therefore, by Remark~\ref{rem:weak-sol}, there exists $w_k\in W^{1,G}_0(B_k)$ being a weak solution to \begin{equation}
\label{eq:muwk'}-\dv \opA(x,Dw_k)=\mu_{w_k}\quad\text{in }\ B_k.
\end{equation} 
Since  $w_k$ is an $\opA$-supersolution, by Comparison Principle (Lemma~\ref{lem:cp}) we get that $w_k\geq 0$. Taking into account the support of $\theta_k$, we notice that $w_k$ is $\opA$-harmonic in $B_k\setminus\overline{\tfrac{5}{4}B_{k+1}}.$ Let us note that\[(w_k-u+\miu)_+\in W^{1,G}_0(B_k).\]
By testing the equations for $u$ and for $w_k$ against this function and then subtracting, we arrive at
\begin{flalign*}
0&\leq \int_{B_k}(w_k-u+\miu)_+\,d\mu_u-\int_{B_k}(w_k-u+\miu)_+\,d\mu_{w_k}\\
&=\int_{B_k}\Big( \opA(x,Du)-\opA(x,Dw_k)\Big)\cdot D(w_k-u+\miu)_+\,dx\\
&=-\int_{B_k\cap \{w_k-u+\miu\geq 0\}}\Big( \opA(x,Du)-\opA(x,Dw_k)\Big)\cdot\big(Du-Dw_k\big)\,dx\leq 0,
\end{flalign*}
where the last inequality follows from the monotonicity of $\opA$. In turn, we directly infer that $D(w_k-u+\miu)_+= 0$ in $B_k$, so
\begin{equation}
    \label{wk-mniejsze}
w_k\leq u-\miu\quad\text{in }\ B_k.
\end{equation}
Take any \[\text{$\phi\in C_0^\infty(B_k)\ $ such that $\ \mathds{1}_{\frac 23 B_k}\leq \phi\leq \mathds{1}_{B_{k}}\ $ and $\ |D\phi|\leq c/R_k$.}\] Then by Maximum Principle (Proposition~\ref{prop:max-princ}) for $w_k$ being $\opA$-harmonic  in $\supp\,D\phi$  \begin{equation}
    \label{wk-in-2/3Bk}
w_k(x)=\min\{w_k(x),\max_{\partial\frac 23 B_k} w_k\}\quad\text{in }\ \supp\,D\phi
\end{equation}
and by Minimum Principle (Proposition~\ref{prop:mini-princ})  \begin{flalign}\label{i-IV-1}\min_{\partial \frac 23 B_k}w_k\leq\min \{w_k(x),\max_{\partial\frac 23 B_k} w_k\}.\end{flalign} Recall that $\min_{\partial\frac 23 B_k} w_k+R_k>0$.
Taking into account~\eqref{muu=muwk}, \eqref{i-IV-1} and extending the domain of integration one can estimate\begin{flalign*}
\mathcal{I}_0&:=\Big(\min_{\partial\frac 23 B_k} w_k+R_k\Big)\mu_u(B_{k+1})\\
&=
\Big(\min_{\partial\frac 23 B_k} w_k+R_k\Big)\mu_{w_k}(B_{k+1})\\
&\leq \int_{ B_k} (\min\{w_k(x),\max_{\partial\frac 23 B_k} w_k\}+R_k)\phi^q\,d\mu_{w_k}(x)=:\mathcal{I}_1\end{flalign*}
Take any $q\geq s_G$. Since $\big((\min\{w_k(x),\max_{\partial\frac 23 B_k} w_k\}+R_k) \phi^q\big)\in W^{1,G}_0(B_k)$, it is an admissible test function in~\eqref{eq:muwk'}, and therefore 
\begin{flalign*}
\mathcal{I}_1&= \int_{ B_k} \opA(x,Dw_k)\cdot D \big((\min\{w_k(x),\max_{\partial\frac 23 B_k} w_k\}+R_k) \phi^q\big)\,dx\\
&= \int_{ B_k\cap \big\{w_k\leq \max_{\partial\frac 23 B_k} w_k\big\}} \opA(x,Dw_k)\cdot D\big( (\min\{w_k(x),\max_{\partial\frac 23 B_k} w_k\}+R_k) \phi^q\big)\,dx\\
&= \int_{ B_k\cap \{w_k\leq \max_{\partial\frac 23 B_k} w_k\}} \opA(x,Dw_k)\cdot D \big((w_k+R_k)\, \phi^q\big)\,dx\\
&=\int_{ B_k\cap \{w_k\leq \max_{\partial\frac 23 B_k} w_k\}} \opA(x,Dw_k)\cdot D w_k\, \phi^q\,dx\\
&\quad +q\int_{ B_k\cap \{w_k\leq \max_{\partial\frac 23 B_k} w_k\}} \opA(x,Dw_k)\cdot D\phi \, \phi^{q-1}\, (w_k+R_k)\,dx
=:\mathcal{I}_2,\end{flalign*}
where we used \eqref{wk-in-2/3Bk}. By the Schwartz inequality, growth conditions, and Lemma~\ref{lem:equivalences}, we get
\begin{flalign*}
\mathcal{I}_2&\leq  \int_{ B_k\cap \big\{w_k\leq \max_{\partial\frac 23 B_k} w_k\big\}} |\opA(x,Dw_k)|\,|D w_k| \phi^q\,dx\\
&\quad + q\int_{ B_k\cap \big\{w_k\leq \max_{\partial\frac 23 B_k} w_k\big\}} |\opA(x,Dw_k)|\, |D\phi|\phi^{q-1}|\min\{w_k(x),\max_{\partial\frac 23 B_k} w_k\}+R_k |\,dx\\
&\leq c \int_{ B_k\cap \big\{w_k\leq \max_{\partial\frac 23 B_k} w_k\big\}} G(|Dw_k|) \phi^q\,dx\\
&\quad + c\int_{ B_k } g(|D(\min\{w_k(x),\max_{\partial\frac 23 B_k} w_k\})|)\, |D\phi|\phi^{q-1}|\min\{w_k(x),\max_{\partial\frac 23 B_k} w_k\}+R_k |\,dx\\
&\leq c \int_{ B_k} G(|D(\min\{w_k(x),\max_{\partial\frac 23 B_k} w_k\})|) \phi^q\,dx\\
&\quad + c\int_{ B_k } g(|D(\min\{w_k(x),\max_{\partial\frac 23 B_k} w_k\})|)\phi^{q-1}\, \Big(\max_{\partial\frac 23 B_k} w_k +R_k\Big)|D\phi| \,dx\\
&=: \mathcal{I}_3^1+\mathcal{I}_3^2.\\
\end{flalign*}
To estimate $\mathcal{I}_3^1$ we note that $\min\{w_k, \max_{\partial \frac 23 B_k}w_k\}$ is $\opA$-supersolution, so \[\Big(\max_{\partial \frac 23 B_k}w_k-\min\big\{w_k, \max_{\partial \frac 23 B_k}w_k\big\}\Big)\quad\text{is a nonnegative $\opA$-subsolution,}\] and hence by the Caccioppoli estimate (Proposition~\ref{prop:caccI}) for this function we get
\begin{flalign*}
\mathcal{I}_3^1&=c \int_{ B_k} G(|D(\min\{w_k(x),\max_{\partial\frac 23 B_k} w_k\})|) \phi^q\,dx\\
&= c\int_{ B_k} G(|D(\max_{\partial\frac 23 B_k}w_k-\min\{w_k(x),\max_{\partial\frac 23 B_k} w_k\})|) \phi^q\,dx\\
 &\leq c \int_{ B_k} G((\max_{\partial\frac 23 B_k}w_k-\min\{w_k(x),\max_{\partial\frac 23 B_k} w_k\}\, )|D\phi|)\,dx \\
 &\leq c \int_{\supp\,D\phi} G\left(\frac{\max_{\partial\frac 23 B_k}w_k}{R_k}\right)\,dx =:\mathcal{I}_4^1,
\end{flalign*}where we used also doubling properties of $G$. As $w_k$ is $\opA$-harmonic on $\supp\,D\phi$, we can use the 
Harnack  inequality (Proposition~\ref{prop:int-sph-harn}) to end with
\[\mathcal{I}_4^1\leq c \int_{\supp\,D\phi} G\left(\frac{\min_{\partial\frac 23 B_k}w_k+R_k}{R_k}\right)\,dx\leq c R_k^n   G\left(\frac{\min_{\partial\frac 23 B_k}w_k+R_k}{R_k}
\right)=:\mathcal{I}_5^1.\]
In order to estimate $\mathcal{I}_3^2$, we use Young's inequality, arguments contained in the proof of Proposition~\ref{prop:caccI}, and Lemma~\ref{lem:equivalences} to get
\begin{flalign*}\mathcal{I}_3^2&=c\int_{ B_k } g(|D(\max_{\partial\frac 23 B_k}w_k-\min\{w_k(x),\max_{\partial\frac 23 B_k} w_k\})|)\phi^{q-1}\, \big(\max_{\partial\frac 23 B_k} w_k +R_k\big)|D\phi| \,dx\\
&\leq c \int_{ B_k } \wt G\left(g\big(|D(\max_{\partial\frac 23 B_k}w_k-\min\{w_k(x),\max_{\partial\frac 23 B_k} w_k\})|\big)\phi^{q-1}\right)\,dx\\
&\quad +c \int_{ B_k } G\bigg((\max_{\partial\frac 23 B_k}w_k+R_k)\, |D\phi|\bigg) \,dx,\\
&\leq c \int_{ B_k }  G\left(|D(\max_{\partial\frac 23 B_k}w_k-\min\{w_k(x),\max_{\partial\frac 23 B_k} w_k\})|\right)\phi^{q}\,dx\\
&\quad +c \int_{ B_k } G\bigg((\max_{\partial\frac 23 B_k}w_k+R_k)\, |D\phi|\bigg) \,dx=:\mathcal{I}_4^2,\end{flalign*}
Then, as in the case $\mathcal{I}_3^1\leq\mathcal{I}_5^1$ above, by Propositions~\ref{prop:caccI} and~\ref{prop:int-sph-harn},  we obtain
\begin{flalign*}\mathcal{I}_4^2&\leq c \int_{\supp\,D\phi} G\left(\frac{\min_{\partial\frac 23 B_k}w_k+R_k}{R_k}\right)\,dx \leq c R_k^n   G\left(\frac{\min_{\partial\frac 23 B_k}w_k+R_k}{R_k}
\right)=c\,\mathcal{I}_5^1.\end{flalign*} 

Summing it up, we get $\mathcal{I}_0\leq \mathcal{I}_3^1 +\mathcal{I}_3^2\leq \bar{c}\, \mathcal{I}_5^1$, which by  Lemma~\ref{lem:equivalences} implies\begin{flalign*}
\mu_u(B_{k+1})&\leq c\, R_k^{n-1} \frac{R_k}{\min_{\partial\frac 23 B_k} w_k+R_k} G\left(\frac{\min_{\partial\frac 23 B_k}w_k+R_k}{R_k}\right)\\
&\leq c\, R_k^{n-1} g\left(\frac{\min_{\partial\frac 23 B_k}w_k+R_k}{R_k}\right).\end{flalign*}
In~\eqref{wk-mniejsze} we noticed that $w_k\leq u-\miu$ in $B_k$, so we have \[
\mu_u(B_{k+1})\leq c\, R_k^{n-1} g\left(\frac{\min_{\partial B_{k+1}}u-\miu+R_k}{R_k}\right)\]
and finally \begin{equation}
\label{g-est}
R_k g^{-1}\left(\frac{\mu_u(B_{k+1})}{R_k^{n-1}}\right)\leq c({\min_{\partial B_{k+1}}u-\miu+R_k} ).
\end{equation}
Having $R_k$ as in~\eqref{RkBk} we estimate
\begin{flalign}\nonumber
\int_0^Rg^{-1}\left(\frac{\mu_u(B(x_0,r))}{r^{n-1}}\right) dr&=\sum_{k=1}^{\infty}\int_{R_{k+1}}^{R_k} g^{-1}\left(\frac{\mu_u(B(x_0,r))}{r^{n-1}}\right) dr\\\nonumber
& \leq\sum_{k=1}^{\infty}\int_{R_{k+1}}^{R_j} g^{-1}\left(\frac{\mu_u(B(x_0,R_k))}{R_{k+1}^{n-1}}\right) dr \\\nonumber
&\leq \sum_{k=1}^{\infty} (R_k-{R_{k+1}}) g^{-1}\left(\frac{4^{n-1}\mu_u(B_{k})}{R_{k-1}^{n-1}}\right)\\
&\leq c\sum_{j=1}^{\infty}\frac{R_{k-1}}{4}\, g^{-1}\left(\frac{\mu_u(B_{k})}{R_{k-1}^{n-1}}\right).\label{series}
\end{flalign}
Therefore by~\eqref{g-est} we get
\begin{flalign*}\int_0^Rg^{-1}\left(\frac{\mu_u(B(x_0,r))}{r^{n-1}}\right) dr&\leq \sum_{k=1}^{\infty}  \frac{c}{4}\,(\min_{\partial B_{R_{k}}}u-\min_{\partial B_{R_{k-1}}}u+R_{k-1})\\
&=c (\lim_{k\to\infty}\min_{\partial B_{R_{k}}}u+R). 
\end{flalign*}
On the other hand, Minimum Principle from Proposition~\ref{prop:mini-princ} for $u$ being $\opA$-supersolution in $B_{k+1}$  yields 
\[\min_{\partial B_{k+1}}u=\min_{B_{k+1}}u\leq  u(x_0)\,.\] 
Therefore, recalling the definition of $\cW^{\mu_u}_G$ given in~\eqref{Wolff-potential}, we can conclude with
\[C_L(\cW^{\mu_u}_G(x,R)-R)\leq u(x_0),\qquad C_L=C_L(\data,n),\]
which ends the proof of the lower bound in~\eqref{eq:u-est}.
\end{proof}

\subsection{Proof of upper bound in Theorem~\ref{theo:est-sol}}\label{ssec:upper-bound}
{Let us remind that in Section~\ref{ssec:reductions} we motivate that it is enough to prove~\eqref{eq:u-est} for $u$ being continuous and bounded $\opA$-supersolutions.}

The main idea of the proof of the upper bound is to modify u to be a weak solution in a countable union of disjoint annuli shrinking to the
reference point $x_0$. For this purpose, we construct a Poisson's modification of $u$ over a family of annuli, see Section~\ref{ssec:Pois} for its basic properties.  The corresponding measure in each annulus concentrates on the boundary of the particular annulus, but in a way we can control, since the measure corresponding to the new solution stays also in the dual of $W^{1,G}(B(x_0,R))$. Since being a
solution is a local property, we are equipped with a priori estimates for weak solutions in each annulus.

\begin{proof}  Our aim now is to compare an $\opA$-supersolution $u$ with its Poisson modification. As previously let us fix $R\in (0,R_{\cW}/2)$, so $B(x_0,2R)\Subset \Omega$. We set again \begin{equation}
    \label{RkBkU} R_k=2^{1-k}R\quad\text{ and }\quad B_k=B(x_0,R_k),\ \ k=0,1,\dots\,.
\end{equation} 

{\em Step 1. Construction of relevant Poisson's modification.} Namely, we modify $u$ in the union of annuli around a~chosen point $x_0$ in the following way. We denote\[\omega=\bigcup_{k=1}^\infty \left((\tfrac{3}{2}B_k)\setminus\overline{B_k}\right).\]
Further, we define a Poisson modification $v=P(u,\omega)$, see~\eqref{PuD}, and we use   Proposition~\ref{prop:Pois} to have that \begin{equation}
    \label{P-u-omega} 
\begin{cases}\text{$v$ is $\opA$-harmonic in }\ \omega,\\ v=u\ \text{ otherwise.}\end{cases}
\end{equation}Note that $v$ is continuous, because of continuity of $u$, and moreover it is  an $\opA$-supersolution  {in $B(x_0,2R)$}. Then by Lemma~\ref{lem:mu-great} there exists $\mu_v\geq 0$, such that \[-\dv\opA(x,D v)=\mu_v\quad\text{ in }\quad B(x_0,2R).\]
Furthermore, we observe that\begin{equation}
\label{muv=muu}\mu_v(B_k)=\mu_u(B_k)\quad\text{ for }\quad  k=0,1,\dots\,.
\end{equation}
Indeed, when we take $\phi\in C_0^\infty(B_k)$ such that $\mathds{1}_{K}\leq\phi\leq \mathds{1}_{B_k}$ for some compact $K\supset \overline{\frac 32 B_{k+1}}$, we get
\[\int_{B_k}\phi\,d\mu_u-\int_{B_k}\phi\,d\mu_v=\int_{B_k}\left(\opA(x,Du)-\opA(x,Dv)\right)\cdot D\phi\,dx=0,\]
where the final equality above results from the fact that $u=v$ on the support of $D\phi$. By exhausting $B_k$ with such $K$, we get~\eqref{muv=muu}.

\bigskip

{\em Step 2. Comparison.} Let us consider $w_k\in W^{1,G}_0(\frac 43 B_{k+1})$ solving\[
-\dv\opA(x,D w_k)=\mu_v\quad\text{ in }\quad \tfrac 43 B_{k+1}\qquad\text{for }\ k=0,1,\dots\,.\]
The measure $\mu_v$ belongs to $(W^{1,G}_0(\frac 43 B_{k+1}))'$, so such $w_k$ exists for every $k$ and  \begin{equation}
    \label{muv=muwk:43B} \mu_{w_k}\Big(\tfrac 43 B_{k+1}\Big)=\mu_v \Big(\tfrac 43 B_{k+1}\Big).
\end{equation} Moreover, by Remark~\ref{rem:nonnegativeness} $w_k\geq 0$. Since $w_k$ is $\opA$-harmonic in the neighbourhood of $\partial \frac 43 B_{k+1}$, it takes continuously zero boundary value on $\partial \frac 43 B_{k+1}$. Furthermore, we observe that
\[\Big(v-\max_{\partial \frac 43 B_{k+1}}v-w_k\Big)_+\in W^{1,G}_0(\tfrac 43 B_{k+1}),\]
so it can be used as a test function in equations for $v$, as well as for $w_k$. By subtracting them we get
\[\int_{\big\{v-\max_{\partial \frac 43 B_{k+1}}v\geq w_k\big\}}\left(\opA(x,Dv)-\opA(x,Dw_k)\right)\cdot D(v-w_k)\,dx=0.\]
Consequently, by the strong monotonicity of the operator, we get that \[D(v-\max_{\partial \frac 43 B_{k+1}}v-w_k)_+=0\] and therefore \begin{equation}
v-\max_{\partial \frac 43 B_{k+1}}v< w_k\label{wk>v}\quad\text{in }\ \tfrac 43 B_{k+1}\,.
\end{equation} 

Since $\min_{\partial \frac 43 B_{k+2}}w_k{+R_k}>0$ near $\partial \frac 43 B_{k+1}$, we have
{\begin{equation}\mu_{\min\big\{w_k,\min_{\partial \frac 43 B_{k+2}}w_k{+R_k}\big\}}\Big(\tfrac 43 B_{k+1}\Big)=\mu_{w_k}\Big(\tfrac 43 B_{k+1}\Big)=\mu_v \Big(\tfrac 43 B_{k+1}\Big).\label{muwk=muv}
\end{equation}}
Indeed, let us notice that
$$
\mu_{\min\big\{w_k,\min_{\partial \frac 43 B_{k+2}}w_k{+R_k}\big\}}\Big(\tfrac 43 B_{k+1}\Big) =\sup_{K\subset \frac 43 B_{k+1}} \mu_{\min\big\{w_k,\min_{\partial \frac 43 B_{k+2}}w_k{+R_k}\big\}} (K),
$$
{where the supremum is taken over compact sets $K$.} 
Then taking any $\phi_K\in C_0^\infty(\frac 43 B_{k+1})$, such that $\mathds{1}_K\leq\phi_K\leq \mathds{1}_{\frac 43 B_{k+1}}$, we get
\begin{flalign*}
\sup_{K\subset \frac 43 B_{k+1}}  \mu&_{\min\big\{w_k,\min_{\partial \frac 43 B_{k+2}}w_k{+R_k}\big\}} (K)\\
&=\sup_{K\subset \frac 43 B_{k+1}}  \int_{\frac 43 B_{k+1}} \phi_K\, d\mu_{\min\big\{w_k,\min_{\partial \frac 43 B_{k+2}}w_k{+R_k}\big\}} \\
&= \sup_{K\subset \frac 43 B_{k+1}} \int_{\frac 43 B_{k+1}} \opA\left(x, D \Big(\min\big\{w_k,\min_{\partial \frac 43 B_{k+2}}w_k{+R_k}\big\}\Big)\right) \cdot D \phi_K\,dx\\
&= \sup_{K\subset \frac 43 B_{k+1}} \int_{\frac 43 B_{k+1}} \opA(x, D w_k)  \cdot D \phi_K\,dx\\
&= \sup_{K\subset \frac 43 B_{k+1}} \int_{\frac 43 B_{k+1}} \phi_K\, d\mu_{w_k}= \mu_{w_k} \Big( \tfrac 43 B_{k+1} \Big). \end{flalign*} 
We justify the second equality above by the fact that  $w_k\in W^{1,G}_0(\tfrac 43 B_{k+1})$, whereas the third one by recalling that $w_k$ takes zero boundary value on $\partial \tfrac 43 B_{k+1}$ also continuously. When we sum it up, to get~\eqref{muwk=muv} it only suffices to remind~\eqref{muv=muwk:43B}.

Further, by~\eqref{muwk=muv}, remembering that $R_k$ is a positive constant, and by approximating $w_k$ by an admissible test function 
we obtain that
\begin{flalign*}
\mathcal{J}_0&:=\Big(\min_{\partial \frac 43 B_{k+2}}w_k+R_k\Big)\mu_v\Big(\tfrac 43 B_{k+1}\Big)\\&=\Big(\min_{\partial \frac 43 B_{k+2}}w_k+R_k\Big)\mu_{\min\big\{w_k,\min_{\partial \frac 43 B_{k+2}}w_k{+R_k}\big\}}\Big(\tfrac 43 B_{k+1}\Big)\quad  \\
&\geq \int_{\frac 43 B_{k+1}}{\Big(\min\big\{w_k,\min_{\partial \frac 43 B_{k+2}} w_k\big\}+R_k\Big)}\,d\mu_{\min\big\{w_k,\min_{\partial \frac 43 B_{k+2}} w_k {+R_k}\big\}}\\
&\geq \int_{\frac 43 B_{k+1}}{\Big(\min\big\{w_k,\min_{\partial \frac 43 B_{k+2}} w_k\big\}\Big)}\,d\mu_{\min\big\{w_k,\min_{\partial \frac 43 B_{k+2}} w_k {+R_k}\big\}}\\
&  {=} \int_{\frac 43 B_{k+1}}\opA\Big(x,D {\Big(}{\min\big\{w_k,\min_{\partial \frac 43 B_{k+2}}w_k \big\} {\Big)}}\Big)\cdot D {\Big(}{{\min\big\{w_k,\min_{\partial \frac 43 B_{k+2}}w_k \big\}}} {\Big)}\,dx\\
& = \int_{\frac 43 B_{k+1}}\opA\Big(x,D\Big({\min\big\{w_k,\min_{\partial \frac 43 B_{k+2}}w_k\big\}}{+R_k}\Big)\Big)\cdot D{\Big({\min\big\{w_k,\min_{\partial \frac 43 B_{k+2}}w_k\big\}{+R_k}\Big)}}\,dx\\
&=:\mathcal{J}_1.\end{flalign*} 
Using growth conditions, the  Poincar\'e inequality (Proposition~\ref{prop:Poincare}), Minimum Principle (Proposition~\ref{prop:mini-princ}) for $w_k$, we get
\begin{flalign*}\mathcal{J}_1 & \geq c \int_{\frac 43 B_{k+1}}G\left(\left|D\Big({\min\big\{w_k,\min_{\partial \frac 43 B_{k+2}}w_k\big\}}+R_k\Big)\right|\right)\,dx\\
& \geq c \int_{\frac 43 B_{k+2}}G\left(\frac{\min_{\partial \frac 43 B_{k+2}}w_k+R_k}{R_k}\right)\,dx\\
&= c\, \left|\tfrac 43 B_{k+2}\right|\, G\left(\frac{\min_{\partial \frac 43 B_{k+2}}w_k+R_k}{R_k}\right)\\
&= c\, R_k^n\, G\left(\frac{\min_{\partial \frac 43 B_{k+2}}w_k+R_k}{R_k}\right)=:\mathcal{J}_2.\end{flalign*} 
Since $\mathcal{J}_0\geq \mathcal{J}_2$ and $\min_{\partial \frac 43 B_{k+2}}w_k+R_k>0$, Lemma~\ref{lem:equivalences} and  Harnack's inequality (Proposition~\ref{prop:int-sph-harn}) give
\[\mu_v\Big(\tfrac 43 B_{k+1}\Big)\geq c\, R_k^{n-1}\, g\left(\frac{\min_{\partial \frac 43 B_{k+2}}w_k+R_k}{R_k}\right)\geq c\, R_k^{n-1}\, g\left(\frac{\max_{\partial \frac 43 B_{k+2}}w_k}{R_k}\right),\]
which is equivalent to
\[\max_{\partial \frac 43 B_{k+2}}w_k\leq c\, R_k\,g^{-1}\left(\frac{\mu_v\Big(\tfrac 43 B_{k+1}\Big)}{R_k^{n-1}}\right).\]
Therefore, having~\eqref{wk>v} and~\eqref{muv=muu}, we may  conclude  that \[\max_{\partial \frac 43 B_{k+2}}v-\max_{\partial \frac 43 B_{k+1}}v\leq  c\, R_k\,g^{-1}\left(\frac{\mu_u( B_{k})}{R_k^{n-1}}\right).\]
Summing both sides from $k=2$ to infinity, we obtain
\begin{equation}
\label{v-sup}\limsup_{k\to\infty} \max_{\partial \frac 43 B_{k+2}}v\leq \max_{\partial \frac 43 B_{3}}v+c\sum_{k=2}^\infty  R_k\,g^{-1}\left(\frac{\mu_u(B_{k})}{R_k^{n-1}}\right).
\end{equation}
Since $v$ is $\opA$-harmonic and nonnegative in $\tfrac 32 B_3\setminus \overline{B_3}$,  Harnack's inequality (Proposition~\ref{prop:int-sph-harn}),  properties of Poisson's modification (Proposition~\ref{prop:Pois}),  give 
\begin{equation}
    \label{J3<J4}\mathcal{J}_3:= \max_{\partial \frac 43 B_{3}}v\leq c\Big(\min_{\partial \frac 43 B_{3}}v+R\Big)\leq c\Big(\min_{\partial \frac 43 B_{3}}u+R\Big)=c\Big(\inf_{\frac 43 B_{3}}u+R\Big)=:\mathcal{J}_4\,,
\end{equation}
where the last equality is due to Minimum Principle (Proposition~\ref{prop:mini-princ}). We observe now that $\tfrac 43 B_{3}=\tfrac 43 B(x_0,2^{1-3}R)=\tfrac 13 B(x_0,R)$. We fix  $s>0$ {that will} be chosen {later}  and by scaling argument we have 
\begin{flalign*}\inf_{\frac 43 B_{3}}u+R=\Big(\inf_{\frac 13 B(x_0,R)}u^{s}\Big)^{\frac{1}{s}} +R&\leq c\left[\left(\barint_{\frac 16 B(x_0,R)}u^{s}\,dx \right)^{\frac{1}{s}}+R\right]\\
&\leq c\left[ \left(\barint_{B(x_0,R)}u^{s}\,dx \right)^{\frac{1}{s}}+R\right] .\end{flalign*}
Therefore, taking any $s=s_0>0$ admissible in {the weak Harnack inequality from Proposition~\ref{prop:inf-est}} for nonnegative $\opA$-supersolution $u$ we can continue estimates~\eqref{J3<J4}  to get
\begin{flalign}
  \nonumber
\mathcal {J}_4&= c\Big(\inf_{\frac 43 B_{3}}u+R\Big)\leq c\left(\Big(\barint_{ B(x_0,R)}u^{s_0}\,dx\Big)^\frac{1}{s_0} +R\right)\\
&\leq c\Big(\inf_{B(x_0,R)}u+R\Big)=:\mathcal{J}_5\,.  \label{J4<J5}
\end{flalign}
{To estimate} $u(x_0)$  we use {the} lower semicontinuity of $u$, the fact that $u=v$ in $\Omega\setminus\omega$ due to~\eqref{P-u-omega}, Comparison Principle (Lemma~\ref{lem:cp}), estimate \eqref{v-sup}, and $\mathcal{J}_3\leq\mathcal{J}_{5}$ by~\eqref{J3<J4}-\eqref{J4<J5} getting
\begin{flalign*}
u(x_0)&\leq \lim_{k\to\infty} \inf_{B_k\setminus \overline{\frac 32 B_{k+1}}} u = \lim_{k\to\infty} \inf_{B_k\setminus \overline{\frac 32 B_{k+1}}} v\\
&\leq \limsup_{k\to\infty} \max_{\partial {\frac 43 B_{k+2}}} v\leq \max_{\partial \frac 43 B_{3}}v+c\sum_{k=2}^\infty  R_k\,g^{-1}\left(\frac{\mu_u( B_{k})}{R_k^{n-1}}\right)\\
&\leq c\Big(\inf_{B(x_0,R)}u+R\Big)+c\sum_{k=2}^\infty  R_k\,g^{-1}\left(\frac{2^{n-1}\mu_u( B_{k})}{R_{k-1}^{n-1}}\right)\\
&\leq c\Big(\inf_{B(x_0,R)}u+R\Big)+c\sum_{k=2}^\infty  (R_{k-1}-R_k)\,g^{-1}\left(\frac{\mu_u( B_{k})}{R_{k-1}^{n-1}}\right).
\end{flalign*}By recalling~\eqref{RkBkU} and the definition of $\cW^{\mu_u}_G$ given in~\eqref{Wolff-potential}, we get \begin{flalign*}
u(x_0)& 
\leq c\Big(\inf_{B(x_0,R)}u+R\Big)+c\int_0^R g^{-1}\left(\frac{\mu_u(B(x_0,r))}{r^{n-1}}\right)\,{dr}\\
&=c\Big(\inf_{B(x_0,R)}u+R+\cW_G^{\mu_u}(B(x_0,R))\Big)\,,
\end{flalign*}
which ends the proof of the upper bound  in~\eqref{eq:u-est}.
\end{proof}
  
\section{The Hedberg-Wolff Theorem} 
\label{sec:W-H}
Having potential estimates from Theorem~\ref{theo:est-sol}, we are in position to prove Theorem~\ref{theo:W-H}.
\begin{proof}[Proof of Theorem~\ref{theo:W-H}] We start with proving implication \eqref{potential-bounded}$\implies$\eqref{mu-in-dual} and later on show the converse. For this proof we need to study solutions to an auxiliary equation involving the special instance of the operator. Namely, we set $\opA(x,\xi)=\xi G(|\xi|)/|\xi|^2$ and thereby we consider\begin{equation}
    \label{eq:main-only-G}-\dv\Big(\frac{G(|Du|)}{|Du|^2}Du\Big)=\mu\quad \text{in }\ \Omega.
\end{equation}

\noindent \textbf{Step 1.} We assume that $\int_\Omega \mathcal {W}_G^\mu(x,R)\,d\mu(x)<\infty$ for some $R>0$ and aim at proving that $\mu\in(W^{1,G}_0(\Omega))'$. We start with $\mu$ supported in a small ball and then we use the partition of unity to cover $\supp\mu$.

 In what follows we suppose that $\supp\,\mu\subset B(x_0, R_0/5)\subset B(x_0,R_0)=:B_0\Subset\Omega$ for some $x_0\in\supp\mu$ and $R_0\in(0,R_\mathcal{W}/2)$ small enough for Theorem~\ref{theo:est-sol} to hold. Further we pick any $R<R_0/5$ and for $j=1,2,\dots$ we define\begin{equation}
    \label{Kjmuj}
K_j=\{x\in\supp\,\mu:\ \cW^\mu_G(x,R)\leq j\}\quad\text{and}\quad \mu_j=\mu\mathds{1}_{K_j}.\end{equation}
Since $\cW^\mu_G$ is lower semicontinuous, sets $K_j$ are compact. For every $j$ there exists a nonnegative $\opA$-superharmonic function $u_j$, such that  $T_k u_j\in W^{1,G}_0(\Omega)$, being a~solution to\begin{equation}
    \label{eq:muj}
-\dv\opA(x,D u_j)=\mu_j\qquad\text{in }\ B_0
\end{equation}
in the approximable sense, see Remark~\ref{rem:approx-sols-exist} for the existence and Remark~\ref{rem:nonnegativeness} for the nonnegativeness. {Our aim is now to prove local boundedness of $u_j$ in $B_0$.}

Note that 
\begin{equation}
    \label{uj-Ah}
\text{$u_j\ $ is $\ \opA$-harmonic in $\Omega\setminus\supp\mu_j$,}
\end{equation} i.e. a solution to $-\dv\opA(x,D u_j)=0$ in this set and it is continuous due to Proposition~\ref{prop:ex-Ahf}. Therefore, $u_j$ is locally bounded in $\Omega\setminus\supp\mu_j$. {In particular, $\sup_{\partial B(x_0,R_0/5)}u_j=:c_j<\infty$.}

By Theorem~\ref{theo:est-sol} we have the existence of $C_U=C_U(\data,n)>0$ such that for $R<R_0/5$
\begin{equation}
    \label{u_j-est-upper-Wolff}
u_j(x)\leq C_U\left(\inf_{B(x,R)} u_j+\cW^{\mu_j}_G(x,R)+R\right)\quad\text{for } \ x\in B(x_0, R_0/5).
\end{equation}

{As a consequence of \cite[Lemma~4.1 (i); (a) and (b)]{CiMa}, there exists a function $\zeta:[0,|\Omega|]\to\rp$, such that $\lim_{s\to 0^+}\zeta(s)=0$ and for every measurable set $E\subset B_0$ it holds that
\[\int_{E}g(u)\,dx< \zeta(|E|)\]
for `approximable solution' $u$ to~\eqref{eq:main-only-G}, see Section~\ref{ssec:sols}. Taking into account that $\mu_j=\mu\mathds{1}_{K_j}$ we infer that
\[\int_{E}g(u_j)\,dx< \zeta(|E|)\] and consequently  there exists $C>0$ independent of $j$ such that
\begin{equation}g^{-1}\left(\barint_{B_0}g(u_j)\,dx\right)< C.
\label{est-CiMa-p}
\end{equation}
Note that for $x\in B(x_0,R_0/5)$ and $R<R_0/5$, we have $B:=B(x,R)\subset B_0$. Therefore, by {the monotonicity of $g$ and} \eqref{est-CiMa-p} we have \begin{flalign}\label{infuj}
\inf_{B }u_j &=g^{-1}\Big(\inf_{ B}g(u_j)\Big) \leq c g^{-1}\left(\barint_{  B}g(u_j)\,dx \right) \\& \leq c g^{-1}  \left(\barint_{B_0}g(u_j)\,dx \right)<C.\nonumber
\end{flalign}}
Furthermore, using~\eqref{u_j-est-upper-Wolff}, \eqref{est-CiMa-p} and~\eqref{infuj}  we get for $x\in B(x_0,R_0/5)$ that
\begin{equation}
    \label{u_j-est-upper-Wolff-C}
u_j(x)\leq C_U\left(C+\cW^{\mu}_G(x,R)+R_0\right)\end{equation}
and consequently, $u_j$ is bounded in $B$ and because of~\eqref{uj-Ah} it is also locally bounded in $B_0$ and in $\Omega$. Hence, by Lemma~\ref{lem:loc-bdd-superharm-are-supersol} $u_j$ is an $\opA$-supersolution and it belongs to $ W^{1,G}_{loc}(\Omega).$ Then by Lemma~\ref{lem:mu-great} it follows that $\mu_j\in (W^{1,G}_0(\Omega))'$ and, consequently, $u_j\in W^{1,G}_0(\Omega).$ 

Since the sequence $\{\mu_j\}$ is nondecreasing (see~\eqref{Kjmuj}), we can deduce the sequence $\{u_j\}$ is nondecreasing  by Comparison Principle (Lemma~\ref{lem:cp}).  Harnack's Principle (Proposition~\ref{prop:har-princ}) ensures in particular that the pointwise limit of $\{u_j\}$ is $\opA$-superharmonic. Note that~\eqref{u_j-est-upper-Wolff-C} gives uniform bound  $\sup_{\partial B(x_0,R_0/5)} u_j<c$. Since $u_j\in W^{1,G}_0(\Omega)$, it can be used as a test function in~\eqref{eq:muj} with $\opA$ as in~\eqref{eq:main-only-G}. Thus by~\eqref{eq:muj}, \eqref{Kjmuj},  \eqref{u_j-est-upper-Wolff-C}, \eqref{uj-Ah} and finally~\eqref{potential-bounded}, we have \begin{flalign*}
\int_{B_0}G(|Du_j|)\,dx&=\int_{B_0}u_j\,d\mu_j(x)\leq \int_{B_0}u_j\,d\mu(x)\\&\leq \int_{B(x_0,R_0/5)}u_j\,d\mu(x)+ \int_{B_0\setminus \supp\mu_j}u_j\,d\mu(x)\\&\leq c \left(\int_{B(x_0,R_0/5)} \cW^{\mu}_G(x,R) \,d\mu(x)+1\right)<\infty,
\end{flalign*}
where $c=c(\data,n)>0$ does not depend on $j$. Since $G\in\Delta_2\cap\nabla_2$, Lemma~\ref{lem:doubling-norm} implies that the sequence of norms is also uniformly bounded, that is there exists $c=c(\data,n)$ such that
\begin{equation}
    \label{unif}\sup_j \| \,|Du_j|\, \|_{L^G(B_0)}=c<\infty.
\end{equation}
By testing~\eqref{eq:muj} with operator $\opA(x,\xi)=\xi G(|\xi|)/|\xi|^2$ 
against arbitrary $\eta\in C_0^\infty(\Omega)$, H\"older's inequality~\eqref{in:Hold}, Lemma~\ref{lem:equivalences}, and~\eqref{unif} we get \begin{flalign*}
\left|\int_\Omega \eta\,d\mu\right|&=\limsup_{j\to\infty}
\left|\int_{B_0} \eta\,d\mu_j\right|=\limsup_{j\to\infty}
\left|\int_{B_0} \frac{G(|Du_j|)}{|Du_j|^2}Du_j\cdot D\eta\,dx\right|\\
&\leq 2\limsup_{j\to\infty}
\left\|\frac{G(|Du_j|)}{|Du_j|}\right\|_{L^{\wt{G}}(B_0)}\| D\eta\|_{L^G(B_0)}\\
&\leq c\limsup_{j\to\infty}\left\|\,|Du_j|\,\right\|_{L^{{G}}(B_0)}\| \,|D\eta|\, \|_{L^G(B_0)}\\
&\leq c \|\eta\|_{W^{1,G}(\Omega)}.
\end{flalign*}
Thus $\mu\in (W^{1,G}_0(\Omega))'$ when $\supp\mu\subset B(x_0,R_0/5).$

In the general case, we use partition of unity to decompose
\[\mu=\sum_{i=1}^{i_0} \mu^{(i)},\qquad i_0<\infty,\]
where for every $i$ it holds $\mu^{(i)}\geq 0$ and $\supp\mu^{(i)}\subset B(x_i, R_0/5).$ Since
\[\int_\Omega\cW^{\mu^{(i)}}_G(x,R)\,d\mu^{(i)}\leq \int_\Omega\cW^{\mu}_G(x,R)\,d\mu <\infty,\]
we obtained above that $\mu^{(i)}\in (W^{1,G}_0(\Omega))'$ for every $i$. Hence, we have the final conclusion that $\mu\in (W^{1,G}_0(\Omega))'$.

\noindent \textbf{Step 2.} We assume that $\mu\in(W^{1,G}_0(\Omega))'$ and justify that then it follows $ \int_\Omega \mathcal {W}_G^\mu(x,R)\,d\mu(x)<\infty$.

Remark~\ref{rem:weak-sol} explains that for such a measure there exists a unique weak solution $u\in W^{1,G}_0(\Omega)$ to equation $-\dv\opA(x,Du)=\mu$ in $\Omega$. Therefore, we can test the solution against $u$ to get
\[\int_\Omega u\,d\mu=\int_\Omega\opA(x,Du)\cdot Du\,dx =\int_\Omega G(|Du|)\,dx<\infty.\]
Since $\cW^\mu_G$ is lower semicontinuous, it is also $\mu$-measurable. If now $R<\tfrac 12 R_\mathcal{W}$, then for fixed $x\in\supp\mu$, by lower bound from Theorem~\ref{theo:est-sol} we have that there exists $c=c(\data,n),$ such that\[\mathcal{W}^\mu_G (x , R)\leq c\big(u(x)+R\big).\]
We conclude the proof upon integrating both sides of the above estimate. 
\end{proof}{}
\section*{Acknowledgements} I. Chlebicka is supported by NCN grant no. 2019/34/E/ST1/00120. A. Zatorska-Goldstein is supported by NCN grant no. 2019/33/B/ST1/00535.

\end{document}